\documentclass[11pt,reqno]{amsart}
\usepackage{amsthm}
\usepackage[normalem]{ulem}
\usepackage{amsfonts}
\usepackage{mathrsfs}
\usepackage{amsmath}
\usepackage{amssymb, amsmath}
\usepackage{color}
\usepackage{graphicx}
\usepackage{dsfont}

\newtheorem{Theorem}{Theorem}[section]
\newtheorem{Proposition}{Proposition}[section]
\newtheorem{Lemma}{Lemma}[section]

\newtheorem{Remark}{Remark}[section]

\numberwithin{equation}{section}

\def\eqdefa{\buildrel\hbox{\footnotesize def}\over =}

\newcommand{\beq}{\begin{equation}}
\newcommand{\eeq}{\end{equation}}
\newcommand{\ben}{\begin{eqnarray}}
\newcommand{\een}{\end{eqnarray}}
\newcommand{\beno}{\begin{eqnarray*}}
\newcommand{\eeno}{\end{eqnarray*}}
\newcommand{\bali}{\begin{aligned}}
\newcommand{\eali}{\end{aligned}}

\newcommand{\ve}{\epsilon}

\newcommand{\ud}{\mathrm{d}}

\newcommand{\ee}{\mathbf{e}}

\newcommand{\RR}{\mathbb{R}}

\newcommand{\CR}{\mathcal{R}}
\newcommand{\CG}{\mathcal{G}}
\newcommand{\CU}{\mathcal{U}}

\newcommand{\CA}{\mathcal{A}}

\newcommand{\CH}{\mathcal{H}}
\newcommand{\CM}{\mathcal{M}}

\newcommand{\BS}{{\mathbb{S}^2}}

\newcommand{\pa}{\partial}

\def\({\left (}
\def\){\right )}
\def\<{\left\langle}
\def\>{\right\rangle}

\def\p{\partial}
\def\O{\Omega}

\begin{document}

\title[Small Deborah number limit of Doi-Onsager equation]{The small Deborah number limit of the Doi-Onsager equation without hydrodynamics}

\author{Yuning Liu }
\address{NYU-ECNU Institute of Mathematical Sciences at NYU Shanghai, 3663 Zhongshan Road North, Shanghai, 200062, P. R. China}
\email{yl67@nyu.edu}

\author{Wei Wang}
\address{Department of Mathematics, Zhejiang University, 310027, Hangzhou, P. R. China}
\email{wangw07@zju.edu.cn}

\date{\today}

\vspace{-0.3in}
\begin{abstract}
We study the small Deborah number limit of the Doi-Onsager equation for the dynamics of nematic liquid crystals without hydrodynamics. This is a Smoluchowski-type equation that characterizes the evolution of a number density function, depending upon both particle position $x\in \mathbb{R}^d(d=2,3)$ and orientation vector $m\in\mathbb{S}^2$ (the unit sphere). We prove that, when the Deborah number tends to zero, the family of solutions with rough initial data near local equilibria will converge strongly to a local equilibrium distribution prescribed by a weak solution of the harmonic map heat flow into $\mathbb{S}^2$. This flow is a special case of the gradient flow to the Oseen-Frank energy functional for nematic liquid crystals. The key ingredient is to show the strong compactness of the family of number density functions and the proof relies on the strong compactness of the corresponding second moment (or the $Q$-tensor), a spectral decomposition of the linearized operator near the limit local equilibrium distribution, as well as the energy dissipation estimate.
\end{abstract}

\maketitle

\section{Introduction}

\subsection{Mathematical theories of the liquid crystal}

Liquid crystals are matter in a  state which has properties between those of a conventional fluid and  those of a solid crystal. The quintessential property of a liquid crystal is its anisotropy.
One of the most common phases for liquid crystal is the nematic phase, in which the molecules  tend to have the same alignment, but their positions are not correlated.  Nematic liquid crystal can be modeled at different scales employing different  {\it order parameters}, which quantify the anisotropic behavior of the material (see for instance \cite{stephen1974physics}), and the choice of the parameters  leads to different theories.

This paper is concerned with two dynamical descriptions of nematic liquid crystals. The more fundamental theory is a microscopic molecular theory, in which the order parameter is a family of number density function  $f(m,x,t)$ on $\BS$ that describes the density of molecules at point $x\in\RR^d$ at time $t$
having alignment $m\in\BS$.   The
alignment $m$ is an idealized description of the orientation of a hard-rod molecule. In a limit that will be rigorously justified in this paper, the microscopic theory  gives rise to the other theory, which is a macroscopic vector theory, and  in this setting, the information is given by a function $n(x,t)$ taking values in the unit sphere $\BS$. The formula that  bridges  these two theories is the following
special form of the number density function
\begin{equation}\label{eq:1.43}
f(m,x,t) = \frac{1}{Z} e^{ \eta (m \cdot n(x,t) )^2 },
\end{equation}
where $\eta$ depends on a coupling constant in the interaction and $Z$ is the renormalization constant. If $\eta$ is large, this is a probability density that is concentrated near $n(x,t)$.

In the microscopic molecular theory, in order to characterize the static configuration of liquid crystals,    Onsager introduced in \cite{Onsager1949}
  a free energy functional on a given domain $\Omega$ as
\begin{equation}\label{eq:free energy-inhom-intro-0}
\begin{split}
  \mathcal{E}[f]=\int_\Omega\int_\BS \left(f\log{f }
+\frac12f \CU[f]\right)\ud m\ud x.
\end{split}
\end{equation}
The first part in \eqref{eq:free energy-inhom-intro-0} is the entropy, corresponding to the (rotational) Brownian motion that the rod-like molecules undergo, while the second part describes  the interaction energy among them. Here the mean-field potential $\CU[f]$
is defined as
\begin{equation}\label{eq:1.06}
\CU[f]=\int_\Omega\int_\BS{B}( x, m; x', m ')f( x ', m ')\ud x '\ud m',
\end{equation}
where ${B}( x, m; x', m ')=B( x-x'; m, m ')\geq 0$ is a kernel function that measures  the interaction  potential energy  between two molecules with configuration $(m,x)$ and $(m',x')$ respectively.
In Onsager's original setting, $B( x-x'; m, m ')$ is chosen to be $1$ if two molecules with
configuration $(m,x)$ and $(m',x')$  are joint, and  $B( x-x'; m, m ')=0$ if otherwise. This definition is called the hard-core excluded volume interaction potential \cite{Onsager1949}.
In this work, we consider an alternative and more regular  form of $B$ which is proposed in \cite{wang2002kinetic}:
\begin{equation}\label{inter:convolution}
  B(x,m;x',m') =\alpha |m\wedge m'|^2 k_\ve({x-x'}).
\end{equation}
 Here $a\wedge b$ denotes the usual wedge product of two vectors $a,b\in\RR^3$,  and $\alpha$ is a parameter that measures the intensity of the potential. Moreover, \[k_\ve(x)=\frac 1{\ve^{d/2}}k\(\frac x{\sqrt{\ve}}\),\]
  where $k(x)$ is a positive  function that decays  at infinity.
The positive parameter $\ve$ represents the typical interaction distance among molecules,  and $d=2$ or $3$ is the dimension of the ambient space.
The above potential  shares qualitatively the same features as Onsager's original potential, but  it is easier
to study analytically due to its smoothness and decoupled structure with respect to  spatial  variable $x$ and alignment direction $m$.

The system considered in this work is the dynamical equation corresponding to \eqref{eq:free energy-inhom-intro-0}, introduced by Doi \cite{doi1988theory}.  Define the chemical potential as
\begin{align*}
  \mu[f]=\frac {\delta \mathcal{E}[f]}{\delta f}=\log f+\CU[f].
\end{align*}
Then the evolution for the number density function $f=f(m,x,t)$  is governed  by the following Smoluchowski equation:
\begin{align}\label{eqn:full-f}
\frac{\partial{f}}{\partial{t}}=\frac 1
{De}\CR\cdot\big(f\CR\mu[f]\big)-\CR\cdot\big( m\wedge(\nabla v)^T\cdot m{f}\big),
\end{align}
where $\mathcal{R}$ is the rotational gradient operator defined on the unit sphere by
$\mathcal{R}= m \wedge\nabla_{ m }$ (see Section 3).  Moreover, $(\nabla v)^T$ is the transpose of the velocity gradient, and
 $De$ is the Deborah number characterizing the typical relaxation time  which is usually very small.
The fluid velocity $v$ satisfies the following Navier-Stokes type equation
\begin{align}\label{eqn:full-v}
v_t+v\cdot\nabla v=&-\nabla{p}+\nabla\cdot\tau+F^e,\qquad \nabla\cdot v=0.
\end{align}
Here $p$ is the pressure, $\tau$ and $F^e$ are stress and body force respectively given by
\begin{align*}
\tau=2\eta_s D+\frac12\xi_r D:\langle mmmm\rangle_f -\langle mm\wedge\CR\mu\rangle_f,\quad
F^e=-\langle\nabla\mu\rangle_f.
\end{align*}
In this expression $\eta_s, \xi_r$ are  material related constants, $D=\frac12(\nabla v+(\nabla v)^T)$, and
$$\langle(\cdot)\rangle_f\eqdefa \int_{\BS}(\cdot)
f(m)\ud{ m}.$$ We remark that the stress term $\tau$ was introduced by Doi \cite{doi1988theory}, while the body force $F^e$ was first introduced by E and Zhang \cite{EZ}.
We also refer  to  \cite{ZhangZhang2008} for the construction of smooth solution  to the system (\ref{eqn:full-f})-(\ref{eqn:full-v}).

 Another theory for nematic liquid crystal is the aforementioned macroscopic vector theory, which views the material as a continuum. The order parameter that it employs is   a unit-vector field $n(x,t)$,  describing the locally preferred alignment of the molecules near the material point $x$.
 The corresponding distortion energy, which is known as the Oseen-Frank energy,
takes the following form:
\begin{equation}\label{eq:oseen}
  E_{OF}[n]=\tfrac {k_1} 2(\nabla\cdot n)^2+\tfrac {k_2} 2( n{\cdot}(\nabla\wedge n))^2
+\tfrac {k_3} 2|n{\wedge}(\nabla\wedge n)|^2
+\tfrac{k_2+k_4}2\big(\textrm{tr}(\nabla n)^2-(\nabla\cdot n)^2\big),
\end{equation}
where $k_1, k_2, k_3, k_4$ are elasticity constants which are distinct in general.
The first three terms in \eqref{eq:oseen} correspond to the three typical pure deformations, i.e.
splay, twist and bend, while  the last term is a null lagrangian, discovered by Ericksen   \cite{Ericksen1976}.
The analytic properties of  minimizers of \eqref{eq:oseen} under Dirichlet boundary condition was investigated in \cite{HardtKinderlehrerLin1986}.
The Oseen-Frank energy \eqref{eq:oseen}  is reduced to the Dirichlet energy
\begin{equation}\label{eq:harmonic}
E_{OF}[n]=\tfrac {\Lambda} 2|\nabla n|^2,
\end{equation}
when one makes the one-constant approximation: $k_1=k_2=k_3=\Lambda, k_4=0$.
Minimizing \eqref{eq:harmonic} among  mappings  from $\O$ into $\BS$ under certain boundary conditions   leads to  harmonic maps into $\BS$, which  are widely studied in the past few decades, see \cite{LinWang2008} and references therein.

For the purpose of describing the hydrodynamics of liquid crystals,  Ericksen and Leslie \cite{EricksenARMA1962Hydrostatics,LeslieARMA1968}
formulated a hydrodynamical system which is known as  Ericksen-Leslie system. It is a very sophisticated  PDE system which couples a Navier-Stokes equation describing the conservation of momentum
with an evolution equation for the vector field $n(x,t)$. We refer to \cite{MR3518239,LinWang2014} for the recent progresses on the mathematics of this system. When the fluid effect is neglected, i.e., the velocity is $0$, then the Ericksen-Leslie system is reduced to the gradient flow of the Oseen-Frank energy \eqref{eq:oseen}.
Under the aforementioned one-constant approximation, this gradient flow becomes the harmonic map heat flow into $\BS$
  \begin{equation}\label{eq:1.44}
  \p_t n=\Lambda(\Delta n+|\nabla n|^2n),
\end{equation}
 which is well-known and widely studied during the past decades. It is worth mentioning that, even for regular initial data, the (local-in-time) smooth solution to \eqref{eq:1.44} might develop singularity at a finite time and thus in general,  the global-in-time solutions to \eqref{eq:1.44} might only have very limited differentiability. See \cite{LinWang2008} and references therein for the analysis of \eqref{eq:1.44}.

Another  theory for nematic liquid crystal is the Landau-De Gennes theory. Like the vector theory, it views the  material as a continuum. However, the order parameter it uses is a symmetric traceless $3\times3$ matrix $Q$ (usually referred to as the $Q$-tensor), which  can be interpreted  as the second moment of a number density function $f$:
\beno
Q[f](\cdot)=\int_{\BS}\big(m\otimes m-\frac 1 3I_3\big)f(m,\cdot)\ud m.
\eeno
We refer to the book by de Gennes-Prost \cite{DeGennesProst1995} for physics of this theory.

\subsection{From microscopic theories to macroscopic theories for liquid crystals}

Exploring the connections between different theories for liquid crystal flow is a fundamental issue in liquid crystal studies.
Kuzzu-Doi \cite{KD} first derived the Ericksen-Leslie equations and determined the Leslie coefficients  from the Doi-Onsager equation under the small Deborah number limit. However,
the Ericksen stress was missing. E-Zhang \cite{EZ} extended Kuzuu and
Doi's formal derivation to the inhomogeneous case and the Ericksen stress was obtained from an extra introduced body force. Roughly speaking, E and Zhang showed
that the solution $(f,v)$ of (\ref{eqn:full-f})-\eqref{eqn:full-v} with $De=\ve$ has a formal expansion
\beno
&&f=f_0(m\cdot n)+\ve f_1+\cdots,\\
&&v=v_0+\ve v_1+\cdots, \eeno
where $f_0$ is an equilibrium distribution  of the form \eqref{eq:1.43}, and $(v_0(x,t),n(x,t))$ is a solution to  the Ericksen-Leslie system.

In \cite{WZZ-cpam}, Wang-Zhang-Zhang give a first rigorous derivation of the
Ericksen-Leslie system from the Doi-Onsager equation when the Deborah number tends to 0 by using the Hilbert expansion method similar to
\cite{Caf} for the Boltzmann equation. In \cite{wang2015rigorous}, the relation between dynamic $Q$-tensor system
and Ericksen-Leslie system was explored by the same authors. In \cite{HLWZZ}, a systematic way was proposed to model liquid crystals for different phases
based on the molecular theory.

In \cite{WZZ-cpam, wang2015rigorous}, the singular limits are justified within the framework of smooth solutions, which excludes a large class of physical solutions that are not regular at space-time locations where the defects of liquid crystal arise. Thus, it is an important question to explore the relationships between different theories in the framework of weak solutions. At this stage, it is worth mentioning that Golse and Saint-Raymond \cite{GS} justified the limit from the renormalized weak solution of the Boltzmann equation to the Leray weak solution of the Navier-Stokes equations.

Our goal is to justify the small Deborah number limit from  the Doi-Onsager equation (\ref{eqn:full-f})-\eqref{eqn:full-v} to the Ericksen-Leslie system in the framework of weak solutions. In this work, we  shall restrict ourselves to the case without hydrodynamics and then the Ericksen-Leslie system is reduced to \eqref{eq:1.44}. The general case should be a challenging problem, due to the possible lack of monotonicity formulas and maximum principle (see a recent work of Lin and Wang \cite{MR3518239}). On the other hand,  Wang, Wang and Zhang \cite{WWZ} justified the limit from the $Q$-tensor flow to \eqref{eq:1.44} in the framework of weak solutions, where the key ingredient is to establish some monotonicity formulas. In \cite{LiuWang2016}, the authors considered the asymptotic limit of $\ve$  for  critical points and minimizers of the energy functional (\ref{eq:free energy-inhom-intro-0})-(\ref{inter:convolution}), and the one-constant approximation of Oseen-Frank energy is derived in the limit.  See also \cite{taylor2017oseen} for a $\Gamma$-convergence approach where a more general energy than \eqref{eq:harmonic} is obtained in the limit.

\subsection{Main results}

To derive the corresponding vector theory of physical interest, we should take $De\sim\ve$ in \eqref{eqn:full-f}, as in \cite{WZZ-cpam}. For simplicity, we set $De=\ve$ and this leads to
the Doi-Onsager equation without hydrodynamics:
\begin{align}\label{main:eqn}
\frac{\partial{f}}{\partial{t}}=\frac 1{\ve}\CR\cdot\big(\CR f+f\CR \CU_\ve[f]\big),\quad (x, m)\in\RR^d\times\BS,
\end{align}
where $\CU_\ve[f]$ denotes the inhomogeneous interaction potential, given by
\begin{align}\label{main:potential}
    &\CU_\ve[f]=\alpha \int_{\RR^d}\int_\BS |m\wedge m'|^2k(\tfrac{x-x'}{\ve }) f( x ', m ')\ud m'\ud x'.
    \end{align}
Note that a related kinetic  model  for self-propelled particles has been discussed in  \cite{MR3305654,frouvelle2012dynamics}.

It is easy to derive a conservation law  for smooth solution to (\ref{main:eqn}):
\begin{equation}\label{eq:1.19}
  \p_t \int_\BS f\ud m=\frac 1 {De}\int_\BS \CR\cdot\big(f\CR(\log{f}+\CU_\ve [f])\big)\ud m=0.
\end{equation}
For the sake of investigating  the small $\ve$ asymptotic  of the solution to \eqref{main:eqn}, we need to know the equilibrium of the homogeneous energy functional (here  {\it homogeneous} refers to the case when the interaction kernel is  independent of spatial variable $x$):
\begin{equation}\label{maiersaupe}
  \mathcal{E}_0[f]=\int_{\BS}\Big(f(  m,\cdot )
\log f( m,\cdot  )+\tfrac{1}{2}\CU_0 [f] (m,\cdot)f(  m,\cdot  )\Big)\ud m,
\end{equation}
where $\CU_0[f]$ denotes the homogeneous interaction potential
\begin{equation}\label{ms-potential}
  \CU_0[f]( m,\cdot)=\alpha\int_{\BS}|m\wedge m'|^2f( m',\cdot)\ud m'.
\end{equation}
The model (\ref{maiersaupe})-(\ref{ms-potential}) is  the so called Maier-Saupe model, of which equilibrium points have been completely classified in \cite{fatkullin2005critical,MR2164198}.
One of the main results there is that, when $\alpha>7.5$ (this is the parameter region in which the isotropic phase loses stability), all minimizers of $ \mathcal{E}_0[f]$ can be written as
\begin{equation}
   {f}_0(m)=h_{\nu}(m):=\frac 1Ze^{\eta(m\cdot \nu)^2},\quad Z=\int_\BS e^{\eta(m\cdot \nu)^2} \ud m,\label{def:f0}
\end{equation}
for every given $\nu\in\BS$. Here $\eta$ is an increasing function of $\alpha$ that will be discussed in Section \ref{subsection:equilibrium} in details.

In the sequel, we shall always assume $\alpha>7.5$ and  denote $E_0$ by the minimum of $\mathcal{E}_0[f]$:
\begin{equation}\label{eq:1.31}
   E_0:=\inf \mathcal{E}_0[f]=\mathcal{E}_0[h_{\nu}].
\end{equation}
Moreover, we introduce  the inhomogeneous energy functional as well as the chemical potential:
\begin{align}\label{eq:free energy-inhom-intro}
  &\mathcal{E}_\ve[f]=\int_{\RR^d}\int_{\BS}\Big( f( x , m )\log{f( x , m )}
+\frac12f( x , m )\CU_\ve[f]( m , x )-{\frac{E_0}{4\pi}}\Big)\ud m\ud x,\\
  &\mu_\ve[f]=\frac {\delta \mathcal{E}_\ve[f]}{\delta f}=\log f+\CU_\ve [f].\nonumber
\end{align}
 For a unit-norm vector field $\nu=\nu(t,x)$, we call $h_{\nu}$ a {\em local equilibrium distribution} (of the energy functional $\mathcal{E}_\ve$).
If $\nu\equiv e_0$ for some fixed $e_0\in\BS$, we call $h_{\nu}$ a {\em uniform equilibrium distribution}.
Local and uniform equilibrium distributions will play analogous roles in our analysis as local and uniform Maxwellians do in the hydrodynamic limit of the Boltzmann equation.

In the sequel, we denote $f_{e_0}:=h_{e_0}(m)$  the uniform equilibrium distribution  oriented by a constant vector $e_0\in\BS.$
Then one has the following energy dissipation law for smooth solution of (\ref{main:eqn}) that decays sufficiently fast to $f_{e_0}$ at $x=\infty$ :
\begin{equation*}
 \frac{\ud }{\ud t} \mathcal{E}_\ve[f]+\frac 1{{\ve}} \int_{\RR^d}\int_{\BS}f|\CR (\log{f}+\CU_\ve[f]) |^2\ud m\ud x= 0.
\end{equation*}
The above identity can formally be derived by first multiplying \eqref{main:eqn} with $\log f+\CU_\ve[f]$ and then integrating by  parts.

The main result of this paper is given below:
\begin{Theorem}\label{thm:converge}
  Consider  $f^{in}_\ve\in L^\infty(\RR^d;L^2(\BS) )$ with
\begin{equation} \label{assump:bound}
    \begin{split}
      &f^{in}_\ve\geq \delta>0,~a.e. ~(x,m)\in\RR^d\times\BS,~\text{for some fixed}~\delta>0,\\
      &\|f^{in}_\ve-f_{e_0}\|_{L^2(\RR^d\times \BS)}<\infty, ~  \int_\BS f^{in}_\ve(x,m)\ud m=1,~a.e. ~x\in\RR^d.
    \end{split}
  \end{equation}
  Then we have

    (i). The Doi-Onsager equation \eqref{main:eqn} with initial condition $f|_{t=0}=f_\ve^{in}$  has a unique positive  solution, denoted by $f_\ve$,  satisfying, for every $T\in (0,\infty)$,
 \begin{equation}\label{eq:1.11}
\begin{split}
&f_\ve\in L^\infty(\RR^d;C^{ \infty }((0,T)\times\BS)),~\p_tf_\ve,\,\Delta_{\BS} f_\ve\in L^\infty(\RR^d;L^2(0,T;H^{-1}(\BS))),\\
&f_\ve\geq C(\ve,T) \delta,~\int_\BS f_\ve(m,x,t)\ud m=1~a.e. ~(x,t)\in\RR^d\times [0,T],
\end{split}
\end{equation}
where $C(\ve,T)$ denotes a positive constant depending on $\ve$ and $T$. Moreover, the following energy dissipation law holds for almost every $t\in (0,T)$:
\begin{align}\label{eq:ME}
&\frac 1\ve\int_{\RR^d} (\mathcal{E}_0[f_\ve](x,t)-\mathcal{E}_0[f_{e_0}] (x,t))\ud x+\frac \alpha{4\ve}\int_{\RR^d\times\RR^d}\left| Q [f_\ve]( x,t )- Q [f_\ve]( y,t)\right|^2
k_\ve\( x - y\)\ud x \ud y\nonumber\\
&\quad+\frac{1}{\ve^2}\int_0^t\int_{\RR^d}\int_{\BS}f_\ve|\CR\mu_\ve [f_\ve]|^2\ud m \ud x \ud \tau=\frac{\mathcal{E}_\ve[f^{in}_\ve]}{\ve},
\end{align}
if the right hand side is bounded.

  (ii). If in addition to \eqref{assump:bound}, assumes that $k(x)\in L^1(\RR^d;\RR_+)$ is a radial function satisfying:
\begin{equation}\label{eq:1.42}
   {|x|^2k(x)\in L^1(\RR^d),~\nabla k(x)\in L^1(\RR^d) }
\end{equation}
and there exists some constant $C>0$ independent of $\ve$ such that
 \begin{equation} \label{assump:uniform}
\|f^{in}_\ve-f_{e_0}\|_{L^2(\RR^d\times \BS)}\leq C,~{\mathcal{E}_\ve[f_\ve^{in}]} \leq  C \ve, ~\text{and}~ \|f_{\ve}^{in}\|_{L^\infty(\RR^d;L^2(\BS))}^2\leq C\ve^{-1},
  \end{equation}
 then up to the extraction of a subsequence, it holds that for every $T>0$ and every compact set $W\subseteq \RR^d$,
\begin{align*}
Q[f_\ve]\xrightarrow{\ve\to 0}  Q[f_0]\quad &\text{strongly in}\,\, C([0,T];L^2(W)),\\
f_\ve \xrightarrow{\ve\to 0}  f_0\quad &\text{strongly in}\,\, L^2\big(W\times \BS \times(0,T)\big),
\end{align*}
where $f_0=\frac 1Z e^{\eta(m\cdot n(x,t))^2}$ for some $n(x,t)\in C([0,T];L^2_{loc} (\RR^d;\BS))$  with
\begin{equation}\label{defweak001}
{n(x,t)-e_0 \in L^\infty\big(0,T;H^1(\RR^d)\big)},~\text{and}~\p_t n\in  L^2 ((0,T)\times\RR^d).
\end{equation}
Furthermore, $n(x,t)$   is a weak solution to \eqref{eq:1.44}
with initial data $n(x,0)$ satisfying
 \begin{equation*}
  \lim_{\ve\to 0}Q[ f_\ve^{in}]=S_2(n(x,0)\otimes n(x,0)-\tfrac 13I_3)~\text{strongly in}~L^2_{loc}(\RR^d).
 \end{equation*}
 Here  $\Lambda$ and $S_2$ are positive constants only depending on the interaction intensity $\alpha$, the dimension $d$, and the kernel function $k(x)$.
\end{Theorem}

\begin{Remark}
A weak solution to \eqref{eq:1.44} is some $n(x,t):\RR^d\times (0,T)\mapsto\BS $ fulfilling \eqref{defweak001} and the following identity for any $\Theta(x)\in C^\infty_c(\RR^d;\RR^3)$ and $\varphi(t)\in C_c^\infty(\RR_+;\RR)$:
\[
  \int_{\RR^d\times\RR_+} (\partial_t n\wedge n)\cdot \Theta(x)\varphi(t)\ud x\ud t =\Lambda \int_{\RR^d\times\RR_+} \varphi(t)\partial_j\Theta(x)\cdot( n\wedge\partial_jn)\ud x\ud t.
\]
It can be verified using $|n(x,t)|\equiv 1$ that, if a weak solution $n(x,t)$ is smooth, then it fulfills $$(\p_t n-\Lambda \Delta n)\wedge n=0$$ and this is equivalent to \eqref{eq:1.44}.
\end{Remark}

\begin{Remark}\label{remark:thm1}
The first part of Theorem \ref{thm:converge} is concerned with the wellposedness of \eqref{main:eqn}, which is proved in the beginning of  Section \ref{sec4}.
  Although these issues can be discussed  under much more relaxed assumptions on the interaction potential \eqref{eq:1.06} as well as the initial data,
  for the sake of investigating the scaling limit, we restrict ourselves to  the inhomogeneous Maier-Saupe potential  defined by \eqref{main:potential}
  and  initial data {\it near the local equilibria}, which include local equilibrium distributions as especial cases.  More precisely, if $n_\ve(x):\RR^d\mapsto\BS$ fulfills
\[ \|n_\ve-e_0\|_{H^1(\RR^d)}\leq C\]
  for some $C$ independent of $\ve$ and for some $e_0\in\BS$, then $f_\ve^{in}(m,x)=\frac 1Z e^{\eta(m\cdot n_\ve(x))^2}$ satisfies \eqref{assump:uniform}.
\end{Remark}
\begin{Remark}
We will give a more detailed discussion on assumptions \eqref{eq:1.42}  in Section 2.1.
\end{Remark}
Now we sketch the key steps in the proof for Part $(ii)$ of Theorem \ref{thm:converge}.
\smallskip

First of all, we will derive the uniform modulated energy estimate for the local energy dissipation \eqref{eq:ME}. This will be the main task of section \ref{sec4} and the primary  difficulty is  how to take care of the integrability of various terms.
Note that the second  condition in  \eqref{assump:uniform} is in analogy  to the {\it relative entropy condition} in \cite{MR2517786}.

The second step is to show that for every $T>0$ and compact domain $W\subseteq \RR^d$,
\begin{equation}
     {f_\ve}\xrightarrow{\ve\to 0} {f}_0\nonumber \quad \text{weakly in }\,\,L^1\big(\RR^d\times \BS\times (0,T)\big)
\end{equation}
for some local equilibrium distribution ${f}_0(m,x,t)=\frac 1Z e^{\eta(m\cdot n(x,t))^2}$.
This is a consequence of \eqref{eq:ME}.
To strengthen the above convergence,  we then prove the strong compactness of the second moment of $f_\ve$:
\begin{equation*}
 Q[f_\ve]\xrightarrow{\ve\to 0} Q[f_0]\quad \text{strongly in }~C([0,T];L_{loc}^2(\RR^{d})).
\end{equation*}
 More precisely, we shall make use of the second term on the left hand side of \eqref{eq:ME} to establish the following uniform  estimates for $Q[f_\ve]$:
\begin{align*}
&\big\|\p_t (Q[f_\ve]*k_\ve)\big\|_{L^2(\RR^d\times (0,T))} \leq C,\\
&\sup_{0\leq t\leq T}\frac 1\ve\int_{\RR^d} \big|Q[f_\ve]*k_\ve-Q[f_\ve]\big|^2 \ud x+\sup_{0\leq t\leq T}\int_{\RR^d} |\nabla (Q[f_\ve]*k_\ve)|^2 \ud x\leq  C.
\end{align*}
This is in a  spirit  similar to   the averaging type lemma in hydrodynamical limit theories of the Boltzmann equation.
In addition, several facts about the critical points of the Maier-Saupe energy (see Section 2) will also play important roles.

The most difficult step is to show that $n(x,t)$ satisfies the harmonic map heat flow. This could be  derived formally through the asymptotic expansion of \eqref{main:eqn} in terms of $\ve$ and a rigorous justification using Hilbert expansion is done in \cite{WZZ-cpam}. Our approach is based on moment method, that is, to consider the limit of the following formulation
\begin{align*}
\int\p_tf_\ve(m,x,t)\psi(m,x,t) \ud m\ud x\ud t
= \frac 1\ve\int\CR\cdot(f_\ve \CR \mu_\ve[f_\ve])\psi(m,x,t)\ud m\ud x\ud t
\end{align*}
for any $\psi(m,x,t)\in  \ker\mathcal{G}^*_{f_0}$. Here $\mathcal{G}_{f_0}=-\mathcal{A}_{f_0}\mathcal{H}_{f_0}$ is the linearized operator of $\CR\cdot\big(f\CR\mu_0[f]\big)$ ($\mu_0:=\log f+\CU_0 [f]$) at
the limiting equilibrium distribution $f_0$, where
\begin{equation}
    \CA_{f_0}\phi=-\CR\cdot(f_0\CR\phi),\quad \CH_{f_0}g=\frac g{f_0}+\CU_0[g].\nonumber
\end{equation}
Owning to $ \ker\mathcal{G}^*_{f_0}= \operatorname{span}\big\{\mathcal{A}_{f_0}^{-1}\CR_if_0\big\}$,  we will take  $$ \psi(m,x,t)=\varphi(t)\CA_{f_0}^{-1}\big(\Theta(x)\cdot\CR f_0\big)$$
for some test function $\Theta(x)\in C_c^\infty(\RR^d)$ and $\varphi(t)\in C_c^\infty(\RR_+)$.
Then the following limit is relatively easy:
\beno
 \lim_{\ve\to0}  \int_{\BS\times\RR^d\times \RR_+}\p_t f_\ve(m,x,t)\psi(m,x,t)\ud m\ud x\ud t={\gamma}\int_{\RR^d\times \RR_+}(\partial_t n\wedge n)\cdot \Theta(x)\varphi(t)\ud x\ud t
 \eeno
for some $\gamma=\gamma(\alpha)\neq 0$.
The main challenge  is to prove the following singular limit:
 \beno
 \lim_{\ve\to0}   \frac 1\ve\int_{\BS\times\RR^d\times \RR_+}\CR\cdot(f_\ve \CR \mu_\ve[f_\ve])\psi(m,x,t)\ud m\ud x\ud t
 ={{\gamma}}\Lambda\int_{\RR^d\times \RR_+} \partial_i\Theta\cdot( n\wedge\partial_in){\varphi(t)}\ud x\ud t,
 \eeno
 for some $\Lambda>0$.  To this end, we decompose the term on the left hand side  by
\begin{multline*}
\frac 1\ve\int_{\BS\times\RR^d\times \RR_+}\CR\cdot(f_\ve \CR \mu_\ve[f_\ve])\psi\ud m\ud x\ud t
={-} \frac 1\ve\int_{\BS\times\RR^d\times \RR_+} {\varphi(t)}  \mu_\ve[f_\ve] \Theta\cdot \CR f_\ve \ud m\ud x\ud t\\
+\frac 1\ve\int_{\BS\times\RR^d\times \RR_+}   \mu_\ve [f_\ve]\CR\cdot (f_\ve\CR\psi{+}{\varphi(t)}\Theta(x) f_\ve) \ud m\ud x\ud t.
\end{multline*}
The first part converges to  \[{{\gamma}}\Lambda\int_{\RR^d\times \RR_+} \partial_i\Theta(x)\cdot( n\wedge\partial_in){\varphi(t)}\ud x\ud t,\] as a consequence of the strong compactness for $Q[f_\ve]$.
The second part can be written as
\begin{align*}
\frac 1\ve\int_{\BS\times\RR^d\times \RR_+}& \mu_\ve[f_\ve]
  \CR\cdot \big(f_\ve(\CR\psi{+}{\varphi(t)}\Theta(x))\big)\ud m \ud x\ud t\\
&={-}\frac 1\ve\int_{\BS\times\RR^d\times \RR_+} \sqrt{f_\ve}\CR\mu_\ve [f_\ve]\cdot\big(\CR\psi{+}{\varphi(t)}\Theta(x)\big)\frac{f_\ve-f_0}{\sqrt{f_\ve}}\ud m \ud x\ud t.
\end{align*}
The key ingredient is to show  that this term vanishes as $\ve\to 0$  and this motivates the  Proposition \ref{prop:strong},
which is of independent interest for mean-field limit problems:  for every $T>0$ and compact set $W\subseteq \RR^d$,
\begin{equation}\label{eq:1.46}
  f_\ve  \xrightarrow{\ve\to 0} f_0\quad \text{strongly in}\,\, L^2\big(W\times \BS \times(0,T)\big).
\end{equation}
Motivated by \cite{GS,WZZ-cpam}, the proof is achieved by combining  the dissipation control in \eqref{eq:ME}
together with the coercive estimate of the linearized operator $\mathcal{G}_{f_0}$ as well as the micro-macro decomposition.
Note that the result of type \eqref{eq:1.46}  is not valid in general in hydrodynamic limit for Boltzmann equation.

\medskip

The rest  of the paper will be organized as follows.
In Section \ref{section:pre}, we will introduce some analytic results related to the Maier-Saupe energy.
In Section 3, we present some basic properties of the rotational operator $\CR$ and a nonlocal operator $\mathcal{L}_\ve$ defined via \eqref{def:L}. These properties will be employed repeatedly in the remainder of the work.
In Section 4, we derive the modulated energy estimate and present some uniform estimates for the solution of the Doi-Onsager equation.
In Section 5, we prove the compactness of the second moment via the control of the modulated energy.
In Section 6, we prove the strong compactness of $f_\ve$ via the dissipation control of the modulated energy and the micro-macro decomposition.
Section 7 is devoted to the proof of Theorem \ref{thm:converge}.

\section{The Maier-Saupe energy}\label{section:pre}

We first introduce some notation. For every $3\times 3$ symmetric matrix $M=\{M_{ij}\}_{1\leq i,j\leq 3}$,  the $j$-th row vector will be denoted by   $M^j=\{M_{ij}\}_{1\leq i\leq 3}$. For any two such matrix $M$ and $N$, their inner product will be defined via
 $M:N=M_{ij}N_{ij}$ under Einstein summation convention
and this induces the norm $|M|=\sqrt{M:M}$. When $i$ appears as superscript or subscript, it denotes an integer. On the other hand, we shall also use $i$ to denote $\sqrt{-1}$ when it is multiplied by some quantities.

\subsection{The interaction kernel of Maier-Saupe energy}

Recall that the inhomogeneous Maier-Saupe energy is defined by
\begin{equation}
\begin{split}
 { \mathcal{E}_\ve[f]=\int_{\RR^d\times\BS}\Big( f( x , m )\log{f( x , m )}
+\frac12f( x , m )\CU_\ve[f]( m , x )-\frac{E_0}{4\pi}\Big)\ud m\ud x}.\nonumber
\end{split}
\end{equation}
Here $E_0$  is defined at \eqref{eq:1.31} and is used for renormalization,
\begin{equation}\label{eq:1.12}
\CU_\ve[f]=\int_{\RR^d}\int_\BS{B}( x, m; x', m ')f( x ', m ')\ud x '\ud m'.
\end{equation}
In this paper, we will take the interaction kernel ${B}( x, m; x', m ')$ as follows
\begin{equation}\label{eq:1.14}
  B(x,m;x',m')=\alpha |m\wedge m'|^2k_\ve({x-x'})
\end{equation}
where $k_\ve(x):=\tfrac{1}{\sqrt{\ve}^d}k(\tfrac{ x }{\sqrt\ve})$.
  Since the interaction potential energy between molecules in consideration are nonnegative and isotropic,
it is quite natural to assume that $k(x)$ is a radial, nonnegative function and $\int_{\RR^d}k(x)\ud x =1$.
Furthermore, we  assume \eqref{eq:1.42}. The first assumption in \eqref{eq:1.42} is crucial to deduce the Oseen-Frank energy with bounded coefficients, see \cite{LiuWang2016,WZZ-cpam}.
On the other hand, we  deduce from it   the following  condition which will be employed in the proof of Theorem \ref{thm:converge} in the last section:
\begin{equation}\label{assump2}
 \lim_{\ve\to 0}\frac 1\ve\int_{|x|\geq \frac{\delta}{\sqrt{\ve}}}k(x)\ud x \le  \lim_{\ve\to 0} \frac{1}{\delta^2}\int_{|x|\geq \frac{\delta}{\sqrt{\ve}}}|x|^2k(x)\ud x  =0,~\forall\delta>0.
\end{equation}
If we denote by $\hat{k}(\xi)$ the Fourier transform of $k(x)$, i.e.,
\begin{equation*}
  \hat{k}(\xi)=\int_{\RR^d}k(x)e^{-2\pi i x\cdot \xi}\ud x,
\end{equation*}
then $\hat{k}(\xi)$ is also a radial real-valued function. Moreover, $|\hat{k}(\xi)|\le 1$, $\hat{k}\in W^{2,\infty}$ and
\begin{equation}\label{ftrans1}
  \hat{k}(0)=1 ,\quad \nabla\hat{k}(0)=0,\quad \nabla^2\hat{k}(0)=-\tfrac {4\pi^2\mu} d I_d.
\end{equation}
We note that the first two formula are obvious while for the last one,  using radial symmetry of $k(x)$, we have
\[\nabla^2\hat{k}(0)=-4\pi^2\int_{\RR^d} x\otimes x k(x)\ud x=\beta I_d\] for some $\beta\in\RR,$ and the result follows by taking the trace of the above formula.

The second  assumption in  \eqref{eq:1.42} implies that there is a constant $C_0$ such that $|(1+|\xi|)\hat{k}(\xi)|\le C_0$ for all $\xi\in\RR^d$,
which implies, for $|\xi|>2C_0$, $\hat{k}(\xi)\le 1/2$ and then $\frac{1-\hat{k}(\xi)}{|\xi|^2}\ge \frac 1{8C_0^2}$.
  On the other hand, we have $\lim_{\xi\to 0}\frac{1-\hat{k}(\xi)}{|\xi|^2}=\frac{2\pi^2\mu}{d}>0$ and $\hat{k}(\xi)<1$ for $|\xi|>0$.
Thus, the continuous function $\frac{1-\hat{k}}{|\xi|^2\hat{k}^2}$ is strictly positive for $|\xi|\le 2C_0$.
Consequently, there exists some $c_0>0$ such that
\begin{equation}\label{assump1}
c_0 |\xi|^2\hat{k}^2(\xi)\leq 1-\hat{k}(\xi) ,~\forall\xi\in\RR^d.
\end{equation}
We will use (\ref{assump2})-(\ref{assump1})  rather than  \eqref{eq:1.42} throughout the paper.

Apparently, there are many examples of $k(x)$ satisfying \eqref{eq:1.42} (and then \eqref{assump2}-\eqref{assump1}). For example,
$k(x)=\(\tfrac a\pi\)^{\frac d2}e^{-a |x|^2}~\text{with}~a\in(0,\pi)$ satisfies all conditions.
Actually, since $\hat{k}(\xi)=e^{-\frac{\pi^2 |\xi|^2}a}$, it is not difficult to see that \eqref{assump1} holds with $c_0\leq \frac{\pi^2}a$.
We also remark that our choice of $k$ here weaken the assumptions in our previous work \cite{LiuWang2016} on the static problem.

It is evident that $k_\ve(x):=\tfrac{1}{\sqrt{\ve}^d}k(\tfrac{ x }{\sqrt\ve})$ satisfies
\begin{equation}\label{eq:1.57}
\hat{k}_\ve (\xi)= \hat{k}(\sqrt{\ve}\xi),\quad \forall \xi\in\RR^d.
\end{equation}
Moreover, $k_\ve$ is a mollifier on $\RR^d$ in the sense that
\begin{equation}\label{L:converge}
  \|v*k_\ve\|_{L^p(\RR^d)}\leq \| k_\ve\|_{L^1(\RR^d)}\|v \|_{L^p(\RR^d)}, ~1\leq p\leq \infty\nonumber
\end{equation}
where $*$ denotes the convolution in $\RR^d$ and for every $v\in L^p(\RR^d)$ with $1\leq p<\infty$,
\begin{equation}
 v*k_\ve  \xrightarrow{\ve\to 0}v\quad \text{strongly in}\,\,L^p(\RR^d).\nonumber
\end{equation}

We shall often work with the traceless second moment $Q(\cdot)=Q[f](\cdot)$ of a number density function $f(\cdot,m)$ with $\int_\BS f(\cdot,m)\ud m=1$ and $f(\cdot,-m )=f(\cdot,m)$,
\begin{align*}
Q(\cdot)=Q[f](\cdot)=\int_\BS (m\otimes m-\frac13I_3) f(\cdot, m)\ud m.
\end{align*}
 Moreover, we define a non-local operator for $Q(x)$:
\begin{equation}\label{def:L}
\mathcal{L}_\ve Q=\frac 1\ve \(Q-Q* k_\ve\).
\end{equation}
According to \eqref{eq:1.12} and \eqref{eq:1.14}, it holds that
  \begin{align}\nonumber
    \CU_\ve[f](m,x,t)=&~\alpha\int_{\BS\times \RR^d} f(m',x',t)|m\wedge m'|^2 k_\ve(x-x')\ud x'\ud m'\\
    =&~\alpha\int_{\BS\times\RR^d} f(m',x',t) k_\ve(x-x')\ud x'\ud m'\nonumber\\ \nonumber
    &-\alpha m\otimes m:\int_{\BS\times\RR^d} m'\otimes m'f(m',x',t) k_\ve(x-x')\ud x'\ud m'\\
    =&~\alpha\(\tfrac 23-(m\otimes m):  Q[f]* k_\ve\). \label{eq:UQ}
  \end{align}
Here we used the fact that $\int_\BS f(m,x,t) \ud m=1$.
Similarly, we deduce from \eqref{ms-potential} that
\begin{equation}\label{eq:UQ0}
\CU_0[f]=\alpha\(\tfrac 23-(m\otimes m):Q[f]\).
\end{equation}
Therefore
\begin{equation}\label{relation:U}
  \tfrac 1\ve(\CU_\ve[f]-\CU_0[f])=\alpha(m\otimes m ):{\mathcal{L}_\ve Q[f]}.
\end{equation}

\subsection{Critical points and minimizers of the homogeneous Maier-Saupe energy}\label{subsection:equilibrium}
We recall some results on the critical points of the homogeneous Maier-Saupe energy:
\begin{equation}\label{energy:MS}
  \mathcal{E}_0[f]=\int_{\BS}\(f(  m)
\log f( m)+\tfrac{1}{2}\CU_0 [f]( m)f(  m )\)\ud m,
\end{equation}
where
\[ \CU_0[f]( m)=\alpha\int_{\BS}|m\wedge m'|^2f( m')\ud m'.\]
In view of \eqref{eq:UQ0}, we can also write (\ref{energy:MS}) as
\begin{equation} \label{energy:MS1}
 \mathcal{E}_0[f]=\int_{\BS}f(  m)
\log f( m)  \ud m+\frac \alpha 3-\frac{\alpha}{2}|Q [f]|^2.
\end{equation}
Various analytic results of \eqref{energy:MS} that will be employed in this work has been obtained in \cite{maiersaupe,fatkullin2005critical,MR2164198}.
To state these results, we define a monotonic increasing function $s_2:\RR\mapsto(-\tfrac 12, 1)$ by
\begin{align*}
s_2(\eta)=\frac{\int_{-1}^1(3x^2-1)e^{\eta x^2}\ud x  }{2\int_{-1}^1e^{\eta x^2}\ud x}.
\end{align*}
\begin{Lemma}\label{uniaxial1}
Every axially symmetric distribution $ h_{\nu}(m)=\frac{ e^{\eta(m\cdot \nu)^2}}{\int_\BS e^{\eta(m'\cdot \nu)^2}\ud m'}$ with given $\nu\in \BS, \eta\in\mathbb{R}$ satisfies
  \begin{equation}\label{eq:Qs2}
  Q[h_{\nu}]=s_2(\eta)\(\nu\otimes \nu-\tfrac 13 I_3\).
\end{equation}
Moreover, $s_2(\eta)$ and $\eta$ share the same sign.
\end{Lemma}
\begin{proof}
The proof can be found in \cite[Lemma 6.6]{WZZ-cpam}. For the convenience of the readers,  we   sketched  it here. Assuming $\nu=(0,0,1)^T$ without loss of generality, one can prove (\ref{eq:Qs2}) by showing the components of both sides are equal.  Moreover, from the  identity
\begin{align*}
\int_0^1z(1-z^2)\ud( e^{\eta z^2})+\int_0^1 e^{\eta z^2} \ud ( z(1-z^2))=e^{\eta z^2} z(1-z^2)|_0^1=0,
\end{align*}
we have
\begin{align*}
s_2(\eta)=\frac{\eta\int_0^1(1-z^2)z^2 e^{\eta z^2} \ud z}{\int_0^1  e^{\eta z^2} \ud z},
\end{align*}
which implies that $s_2(\eta)$ and $\eta$ have the same sign.
\end{proof}

In \cite{fatkullin2005critical,MR2164198},   all the  smooth  critical points of (\ref{energy:MS}) are characterized:
\begin{Proposition}\label{prop:critical}
All the  smooth  critical points of (\ref{energy:MS}) are given by
 \begin{equation*}
   h_\nu(m):=\frac{ e^{\eta(m\cdot \nu)^2}}{\int_\BS e^{\eta(m'\cdot \nu)^2}\ud m'}
 \end{equation*}
for  every given $\nu\in\BS$, where $\eta$ and $\alpha$ satisfies the following relation:
\beq\label{eta}
\eta=\alpha s_2(\eta).\nonumber
\eeq
For every $\alpha>0$, $\eta=0$ is  a solution of (\ref{eta}). In addition, defining
$$\alpha^*=\min_{\eta\in\mathbb{R} }\frac{\int_{-1}^1e^{\eta x^2}dx  }{\int_{-1}^1x^2(1-x^2)e^{\eta x^2}dx},$$
we have
\begin{enumerate}
  \item when $\alpha<\alpha^*$, $\eta = 0$ is the only solution of (\ref{eta});
  \item when $\alpha=\alpha^*$, besides $\eta = 0$ there is another solution $\eta=\eta^*$ of (\ref{eta});
  \item  when $\alpha>\alpha^*$, besides $\eta=0$ there are two solutions $\eta_1>\eta^*>\eta_2$ of (\ref{eta}).
\end{enumerate}
\end{Proposition}
Furthermore, the stability/instability of critical points have also been clearly discussed.
\begin{Proposition} \label{prop:critical-stab}
Let $\alpha^*$ be the parameter defined above.
\begin{enumerate}
  \item   When $\alpha<\alpha^*$, $\eta=0$ is the only critical point. Thus, it is stable;
   \item When $\alpha^*\le \alpha<7.5$,  the solution corresponding to $\eta = 0$ and $\eta=\eta_1$ are both stable;
  \item  When $\alpha>7.5$, the solution corresponding to $\eta=\eta_1$ is the only stable solution.
\end{enumerate}
\end{Proposition}
As a consequence of the above results,  we shall choose $\alpha>7.5$ and define
\begin{equation}\label{eq:1.58}
  \eta=\eta_1(\alpha),\quad S_2=s_2(\eta_1(\alpha))
\end{equation}
throughout  this paper.
  In addition, we denote for any $\nu\in\BS$
 \begin{equation}\label{uniaxial}
   h_{\nu}(m):=\frac{1}{Z} e^{\eta(m\cdot \nu)^2},
 \end{equation}
 where $Z=\int_\BS e^{\eta(m\cdot \nu)^2}\ud m$ is a constant independent of $\nu$.
 As remarked in the introduction, the distributions $h_{\nu}$ play analogous  roles that local Maxwellians do  in the hydrodynamic limit of Boltzmann equation.

The following lemma  shows that $h_{\nu}(m)$ are the only global  minimizers of the Maier-Saupe energy (\ref{energy:MS}) in $L^1(\BS)$ when $\alpha>7.5$.
\begin{Lemma}\label{lem:minimizer}
For $\alpha>7.5$, the global minimizers of \eqref{energy:MS} in the function class
\begin{equation}\label{eq:1.69}
\mathscr{H}:=\left\{f\in L^1(\BS)~\mid ,~f\geq 0~\text{a.e. on}~\BS,~\|f\|_{L^1(\BS)}=1\right\}
\end{equation}
are achieved only by  the distributions $h_\nu(\forall\nu\in\BS)$ in \eqref{uniaxial}.
 \end{Lemma}

 \begin{proof}
 The existence of global minimizers follows from the direct method in calculus of variations.
 It remains to show that they are smooth and bounded away from zero and are consequently stable smooth critical points. This together with Proposition \ref{prop:critical} and Proposition \ref{prop:critical-stab} lead to the desired result.

  For any $f\in\mathscr{H}$, the eigenvalues of $Q[f]$ lie in $(-1/3,2/3)$. So it follows from \cite{maiersaupe,HLWZZ} that there exists a traceless symmetric matrix $B(Q)$ such that the probability density defined by
 \begin{equation}\label{modify001}
 f_Q(m):= \frac{e^{B(Q):m\otimes m}}{\int_\BS e^{B(Q):m\otimes m}\ud m}\in \mathscr{H}
 \end{equation}
 satisfies $Q[f_Q]=Q[f]$ and
 \begin{equation*}
 \int_\BS f\log f\ud m\geq \int_\BS f_Q\log f_Q\ud m.
 \end{equation*}
  Together with formula \eqref{energy:MS1}, we infer that $\mathcal{E}_0[f]\geq \mathcal{E}_0[f_Q]$.
   So we have shown that the global minimizers must have the form \eqref{modify001}.
\end{proof}

 We end up this section by  the following compactness result for the sequence of functions with finite entropy. See for instance  \cite{MR2197021} for details of the proof.
\begin{Lemma}\label{lem:convex}
For any {bounded domain} $\O\subset\RR^d$, let $$f_k\in\mathscr{H}(\O):=\left\{f\in L^1(\BS\times\Omega),f(x,m)\geq 0,\|f(\cdot,x)\|_{L^1(\BS)}=1,~a.e.~ x\in\Omega\right\}$$ be a sequence of functions such that
   \begin{equation*}
     \int_{\Omega\times\BS}f_k
\log f_k~ <\infty~\text{uniformly for}~k\in\mathbb{N}^*.
   \end{equation*}
    Then modulo the extraction of a subsequence,   there exists $f\in \mathscr{H}(\O)$ such that $f_k\rightharpoonup f$ weakly in $L^1(\BS\times\Omega)$ and
  \begin{equation*}
  \int_{\Omega\times\BS} f\log f\ud x\ud m
\leq\liminf_{k \to \infty}\int_{\Omega\times\BS}f_k\log f_k\ud x\ud m.
  \end{equation*}
\end{Lemma}

\section{Basic properties of two operators $\CR$ and $\mathcal{L}_\ve$}
  {In what follows, we adopt Einstein summation convention by summing over repeated latin index. In various estimates in the sequel, $C$ will be a generic positive constant which might change from line to line and will be  independent of $\ve$ unless otherwise specified.}
\subsection{Rotational gradient operator $\CR$}
 We first give some basic properties for the rotational gradient operator on the unit sphere $\BS$, which is defined by
$$\mathcal{R}=m \wedge\nabla_{ m },$$
where {$\nabla_m$ is the restriction of standard gradient $\nabla$ on $\BS$}.
Under the spherical coordinate  on $\BS$ with $m=(\sin\theta\cos\phi, \sin\theta\sin\phi, \cos\theta)$, $\mathcal{R}$ can be written explicitly as
\begin{equation}\label{eq:1.75}
  \begin{split}
    \CR=&(-\sin\phi\mathbf{e}_1+\cos\phi\mathbf{e}_2)\partial_\theta
-(\cos\theta\cos\phi\mathbf{e}_1+\cos\theta\sin\phi\mathbf{e_2}-\sin\theta\mathbf{e}_3)
\frac{1}{\sin\theta}\partial_\phi\\
\triangleq&\ee_1{\CR}_1+\ee_2{\CR_2}+\ee_3{\CR_3}.
  \end{split}
\end{equation}
It is straightforward  to verify the following two properties for $\CR$:
\begin{align}\label{identity:R-1}
&  \int_{\BS}\mathcal{R} f_1 f_2 \ud m =-\int_{\BS}f_1\mathcal{R} f_2 \ud m,\\
 & \CR_im_j=-\varepsilon^{ijk}m_k,\qquad \CR\cdot\CR=\Delta_\BS \label{identity:R-2}
\end{align}
where  $\varepsilon^{ijk}$ is the Levi-Civita symbol. Consequently, we can derive from (\ref{identity:R-2}) that
\begin{align}\label{Appd:eq:R1}
 &   \CR(m\cdot u)=m\wedge u,~\CR\cdot(m\wedge u)=-2m\cdot u,\\
  &  \CR(B:m\otimes m)=2 m\wedge(B\cdot m) \label{Appd:eq:R2},\\
  &\Delta_\BS (m\otimes m)=-6\big(m\otimes m-\tfrac13I_3\big)\label{Appd:eq:R3}
\end{align}
for every constant vector $u\in \RR^d$ and  symmetric  matrix $B$.

We infer from (\ref{Appd:eq:R2})-(\ref{Appd:eq:R3}) and \eqref{eq:UQ} that, if $f=f(\cdot,m)$ fulfills $\int_\BS f(\cdot,m)\ud m=1$, then
\begin{align}\label{eq:RU}
  \CR \CU_\ve[f]=&~-\alpha\CR \((m\otimes m):  Q[f]* k_\ve\)=\{-2\alpha m_k m_j\varepsilon^{ki\ell}Q_{ij}[f]*k_\ve\}_{1\leq \ell\leq 3},\\
    \Delta_{\BS} \CU_\ve[f]=&~6\alpha(m\otimes m-\tfrac13\mathbb{ I}_3):  Q[f]* k_\ve. \label{eq:R2U}
\end{align}
In addition, for  $f_0(m)=\frac{1}{Z} e^{\eta(m\cdot n)^2}$, we have
\begin{align}\label{eq:Rf}
\CR f_0=f_0\CR(\log f_0)=\eta f_0\CR (m\cdot n)^2 =2 \eta (m\wedge n)(m\cdot n)f_0.
\end{align}

\subsection{Nonlocal operator $\mathcal{L}_\ve$}

 For any function $u \in L^2(\RR^d)$, we define
\begin{equation}\label{eq:1.34}
  \mathcal{L}_\ve u=\frac 1\ve (u-u*k_\ve).
\end{equation}
Apparently, $ \mathcal{L}_\ve $ is a bounded operator from $L^2(\RR^d)$ to $L^2(\RR^d)$ with operator norm depending on $\ve$.
In addition, $\mathcal{L}_\ve$ is a multiplier operator with non-negative symbol
\begin{equation*}
\widehat{\mathcal{L}_\ve u}(\xi)=\frac{\hat{k}(0)-\hat{k}(\sqrt{\ve}\xi)}{\ve}  \hat{u}(\xi),\quad \forall \xi\in\RR^d.
\end{equation*}
Actually it follows from  \eqref{assump1} that $\hat{k}(0)-\hat{k}( \xi)\geq 0$ for any $\xi\in\RR^d$. As a result, we can define $h(\xi)$ as
\begin{equation}\label{lipschitz}
h(\xi):= \left\{\begin{array}{rl}
  \xi\sqrt{\frac{\hat{k}(0)-\hat{k}( \xi)}{  |\xi|^2}  },&\xi\in\RR^d\backslash\{ 0\},\\
  0,& \xi=0.
\end{array}\right.
\end{equation}

\begin{Lemma}\label{lipschitz1}
  The function $h(\xi)$ defined by \eqref{lipschitz} is globally Lipschitz in $\RR^d$.
\end{Lemma}
\begin{proof}
 It follows from \eqref{ftrans1} that  $h(\xi)$ is continuous at $\xi=0$ since $\lim_{\xi\to 0} h(\xi)=0$. On the other hand, $h(\xi)$ is smooth in $\RR^d\backslash\{0\}$ and decays to zero when $\xi\to \infty$. So $h\in L^\infty(\RR^d)\cap C(\RR^d)$.
 We compute the derivative of $h$ by
 \begin{equation*}
   \begin{split}
     \nabla h(\xi)=\mathbb{I}_d\sqrt{\frac{1-\hat{k}(\xi)}{|\xi|^2}}-\frac{\xi}{2\sqrt{ 1-\hat{k}(\xi) }} \otimes\(\frac{\nabla\hat{k}(\xi)}{|\xi|}+\frac {2\xi}{|\xi|^3}(1-\hat{k}(\xi))\)=\sum_{k=1}^3 A_i(\xi),\forall \xi\neq 0.
   \end{split}
 \end{equation*}
 It is evident that $A_1,A_3\in L^\infty(\RR^d)\cap C(\RR^d)$. Moreover, $A_2\in L^\infty(B_1)\cap C^\infty(\RR^d\backslash B_1)$ and tends to zero as $\xi\to \infty$. These all together imply the statement.
\end{proof}
Therefore, we can decompose $\mathcal{L}_\ve$ as square of two first-order   vector-valued operator $\mathcal{T}_\ve=\{\mathcal{T}^i_\ve\}_{1\leq i\leq d}$ defined by
\begin{equation}\label{defteps}
\widehat{\mathcal{T}_\ve u}(\xi)
=\xi\sqrt{\frac{\hat{k}(0)-\hat{k}(\sqrt{\ve}\xi)}{\ve |\xi|^2}  } \hat{u}(\xi)=\frac1{\sqrt\ve} h(\sqrt\ve \xi)\hat{u}(\xi).
\end{equation}

\begin{Lemma}\label{converge-T}
The operator $\mathcal{L}_\ve$ and $\mathcal{T}_\ve$ are bounded from $L^2(\RR^d)$ to $L^2(\RR^d)$ with operator norm depending on $\ve$ and
\begin{equation}\label{squrtroot}
   \mathcal{L}_\ve  =\sum_{  k=1}^d\mathcal{T}^k_\ve \cdot \mathcal{T}^k_\ve .
 \end{equation}
  Moreover, for every $u\in H^1(\RR^d)$,   it holds
\begin{equation*}
  \mathcal{T}_\ve u \xrightarrow{\ve\to 0} -i{\sqrt{\tfrac \mu {2d}}} \nabla u\quad \text{in }L^2(\RR^d).
\end{equation*}
\end{Lemma}
\begin{proof}
The first statement is due to \eqref{defteps}, Plancherel theorem and Lemma \ref{lipschitz1}. To prove the `moreover' part, it can be verified from \eqref{ftrans1} that
\begin{equation*}
  \sqrt{\frac{1-\hat{k}( \sqrt{\ve}\xi)}{ \ve |\xi|^2}  }~\text{is uniformly bounded with respect to }~\ve>0~\text{and}~\xi\in \RR^d\backslash\{0\},
\end{equation*}
and
\begin{equation*}
  \lim_{\ve\to 0}\sqrt{\frac{1-\hat{k}( \sqrt{\ve}\xi)}{ \ve |\xi|^2}  }={\pi\sqrt{\frac {2\mu} d}},\qquad\quad~\forall~\xi\in\RR^d\backslash\{0\}.
\end{equation*}
On the other hand, as $u\in H^1(\RR^d)$, we have
$$\int_{\RR^d}  (|\xi|^2+1)  |\hat{u}(\xi) |^2\ud\xi<\infty.$$
Therefore, Lebesgue's dominant convergence theorem implies
\begin{equation*}
   \lim_{\ve\to 0}\left\| \(\mathcal{T}_\ve+{i\sqrt{\frac\mu {2d}} \nabla}\)u\right\|^2_{L^2(\RR^d)}
   = \lim_{\ve\to 0}\int_{\RR^d}
   \left|  \left(\sqrt{\frac{1-\hat{k}( \sqrt{\ve}\xi)}{ \ve |\xi|^2}  }{-\pi\sqrt{\frac {2\mu} d}}\right)\xi\hat{u}(\xi) \right|^2\ud\xi =0.
\end{equation*}
\end{proof}

\section{Global wellposedness and uniform  energy estimate}\label{sec4}

 In this section, we study the global existence of solution to (\ref{main:eqn}) and establish the energy dissipation relation \eqref{eq:ME}.
As noted in Remark \ref{remark:thm1}, these issues can be discussed  under much more relaxed assumptions on the interaction potential \eqref{eq:1.06} as well as the initial data, see for instance \cite{frouvelle2012dynamics} for the spatial homogeneous case. However, for the sake of investigating the scaling limit,  we shall restrict ourselves to  the inhomogenous Maier-Saupe potential  defined by \eqref{main:potential} and integrable initial data.

From \eqref{eq:free energy-inhom-intro}, \eqref{ms-potential} and \eqref{relation:U}, we can write
\begin{align}\label{eq:1.68}
\mathcal{E}_\ve[f]
=&\int_{\RR^d\times\BS} \Big(f\log f+\frac 12 f\mathcal{U}_0[f]+\frac{\alpha\ve}2 f (m\otimes m):\mathcal{L}_\ve Q[f]-{\frac{E_0}{4\pi}}\Big)\ud m\ud x\nonumber\\
=&\int_{\RR^d} \Big(\mathcal{E}_0[f]-E_0+\frac{\alpha\ve}2 Q[f]:\mathcal{L}_\ve Q[f]\Big)\ud x\nonumber\\
=&\int_{\RR^d}\Big(\mathcal{E}_0[f]-E_0\Big) \ud x+ \frac {\alpha}{4}\int_{\RR^d\times\RR^d}\big|Q [f]( x )- Q [f]( y)\big|^2
k_\ve\( x - y\)\ud x \ud y.
\end{align}

We also recall the definition \eqref{def:f0} that  $f_{e_0}(m):=\frac 1Z e^{\eta(m\cdot e_0)^2}$  for some fixed $e_0\in\BS$. Without loss of generality we choose $e_0=(0,0,1)$.

\begin{Theorem}\label{thm:wellpose}
For any $f^{in}\in L^\infty(\RR^d;H^s(\BS))$ with $s\geq 0$ and
\begin{equation} 
    f^{in}\geq \delta>0,~\|f^{in}-f_{e_0}\|_{L^2(\RR^d\times \BS)}<\infty, ~  \int_\BS f^{in}(x,m)\ud m=1,~a.e. ~x\in\RR^d
  \end{equation}
  for some fixed constant $\delta>0$,  the Doi-Onsager equation \eqref{main:eqn} with initial condition $f|_{t=0}=f^{in}$  has a unique positive  solution $f$ satisfying,
for any $T\in (0,\infty)$,
 \begin{equation}\label{last3}
\begin{split}
&f\in L^\infty(\RR^d;C^{ \infty }((0,T)\times\BS)),~\p_tf,\,\Delta_{\BS} f\in L^\infty(\RR^d;L^2(0,T;H^{s-1}(\BS))),\\
&f\geq C(\ve,T) \delta,~\int_\BS f(m,x,t)\ud m=1~a.e. ~(x,t)\in\RR^d\times [0,T],
\end{split}
\end{equation}
for some constant $C(\ve,T)>0$.
Moreover, the following  energy dissipation law holds:
\begin{align}\label{push1}
    & \frac{\mathcal{E}_\ve[f]}\ve+\frac{1}{\ve^2}
    \int_0^t\int_{\RR^d}\int_{\BS}f|\CR\mu_\ve [f]|^2\ud m \ud x \ud \tau = \frac{\mathcal{E}_\ve[f^{in}]}\ve,~\text{for a.e.} ~ t\in(0,T),
\end{align}
if the right hand side of \eqref{push1} is finite.
\end{Theorem}

\begin{Remark}\label{remark1}
 This theorem leads to  Part $(i)$ of Theorem \ref{thm:converge}. Here, we also remark that
 the admissible set of initial data satisfying the uniform bound in (\ref{assump:uniform})  includes at least a family of local equilibrium distributions.
More precisely, for  $f^{in}=h_{n^{in}(x)}(m)$ (consequently  $\mathcal{E}_0[f^{in}]=E_0$) with $n^{in}(x)-e_0 \in H^1(\RR^d),$ one can verify that
\begin{equation}\label{eq:1.32}
   \frac \alpha{4\ve}\int_{\RR^d\times\RR^d}\left| Q [f^{in}]( x )- Q [f^{in}]( y)\right|^2
k_\ve\( x - y\)\ud x \ud y\le C,
\end{equation}
with $C$ independent of $\ve$, which combined  with  \eqref{eq:1.68}  implies that
\[0\leq \mathcal{E}_\ve[f^{in}]\leq C\ve.\]
Note that \eqref{eq:1.32}  is due to the following fact: for any $v\in L^2_{loc}(\RR^d)$ with $$\nabla v\in L^2(\RR^d)~\text{and}~ \|v-v_0\|_{L^2(\RR^d)}<\infty$$ for some constant vector $v_0$, it holds that
  \begin{align*}
     &\int_{\RR^d\times\RR^d}\left| v( x )- v( y)\right|^2 k_\ve\( x - y\)\ud x \ud y\\
=&\int_{\RR^d\times\RR^d}\left|v(y+z)-v(y)\right|^2k_\ve\(z\)\ud z\ud y\\
\leq &\int_{\RR^d\times\RR^d}|z|^2\int_0^1\left|\nabla v(yt+(1-t)(y+z))\right|^2\ud t k_\ve\(z\)\ud z\ud y\\
=&\ve\int_0^1\int_{\RR^d}|\frac z{\sqrt{\ve}}|^2k_\ve(z)\(\int_{\RR^d}\left|\nabla v(y+(1-t)z)\right|^2 \ud y\)\ud z\ud t\\
=&\ve\int_0^1\int_{\RR^d}|\frac z{\sqrt{\ve}}|^2 k_\ve(z)\|\nabla  v \|^2_{L^2(\RR^d)}\ud z\ud t\\
=& \ve \|\nabla v \|^2_{L^2(\RR^d)}\int_{\RR^d}|x|^2 k(x)\ud x.
   \end{align*}

\end{Remark}

\begin{proof}[Proof of Theorem \ref{thm:wellpose}]
During the proof,  $C_\ve$ will denote a generic constant, which
 might depend on $\ve$ and might change from line to line. In addition, we  write $f$ instead of $f_\ve$ for brevity.

{\bf Part 1: Existence, uniqueness and regularity.} In this part we shall focus on the wellposedness of \eqref{main:eqn}.
The proof will be divided into several steps, and in Step 2 and Step 3 we follow the method developed  in \cite{frouvelle2012dynamics}.

{\it Step 1: Existence and uniqueness of  solution with $f^{in}\in L^\infty(\RR^d;H^s)$ for any $s\geq 0$.}

The main purpose of this step is to construct a strictly positive solution to \eqref{main:eqn}. To this end,  we first define a nonlinear operator
\[\mathscr{F}g=  \CR\cdot(f\CR\CU_\ve[f])  \]
where $f$ and $g$ are related by
\begin{equation*}
\ve\p_t f=\Delta_\BS f+g,~f\mid_{t=0}=f^{in},
\end{equation*}
as well as the following  function spaces
$${\mathscr Y}_s:=L^\infty(\RR^d;L^2(0,T;H^{s-1}(\BS))),$$
and
\begin{equation}\label{last1}
{\mathscr X}_s:=\{f(m,x,t)\mid  (f_t,\Delta_\BS f)\in L^\infty(\RR^d;L^2(0,T;H^{s-1}(\BS)))\}.
\end{equation}
We  equip $\mathscr X_s$ with  norm $$\|f\|_{\mathscr X_s}=\|(\p_tf ,\Delta_\BS f) \|_{L^\infty(\RR^d;L^2(0,T;H^{s-1}(\BS)))}+\|f\|_{L^\infty(\RR^d;C([0,T];H^s(\BS)))}.$$
We shall also assume in this step that $T<1$.
Then a standard estimate for the heat equation gives
\begin{equation}\label{eq:1.18}
\|f\|_{\mathscr X_s}\leq C_\ve\(\|f^{in}\|_{L^\infty(\RR^d;H^s(\BS))}+\|g\|_{\mathscr Y_s}\).
\end{equation}
It follows from \eqref{eq:1.12}  that, every $f\in \mathscr{X}_s$ fulfills, for every $k\in\mathbb{N}$,
\begin{equation}
   \|(\CU_\ve[f], \CR\CU_\ve [f ], \Delta_\BS\CU_\ve [f ])\|_{L^\infty( (0,T)\times \RR^d; C^k( \BS))}\leq C_\ve\|f\|_{L^\infty( \RR^d\times [0,T];L^1( \BS))}\leq C_\ve\|f\|_{\mathscr X_s},\label{eq:1.61}
\end{equation}
where $C_\ve$ is independent of $f$ and $T>0$.   It follows from \eqref{eq:1.61} that, for almost every $t\in [0,T]$,
 \begin{equation*}
 \begin{split}
    &\| \mathscr{F} g(\cdot,t)\|_{L^\infty(\RR^d;H^{s-1}(\BS))}\leq C_\ve\|f\|_{\mathscr X_s}\|f(\cdot,t)\|_{L^\infty(\RR^d;H^{s}(\BS))}\\
    &\leq C_\ve\|f\|_{\mathscr X_s}\|f(\cdot,t)\|^{1/2}_{L^\infty(\RR^d;H^{s-1}(\BS))}\|f(\cdot,t)\|^{1/2}_{L^\infty(\RR^d;H^{s+1}(\BS))}\\
    &\leq C_\ve \|f \|_{\mathscr X_s}^{3/2} \|f(\cdot,t)\|^{1/2}_{L^\infty(\RR^d;H^{s+1}(\BS))}.
 \end{split}
 \end{equation*}
 This together with \eqref{eq:1.18} and the Cauchy-Schwarz inequality implies
  \begin{equation}
 \begin{split}
    \| \mathscr{F} g\|^2_{\mathscr Y_s}\leq C_\ve \sqrt{T} \|f\|_{\mathscr X_s}^{4}\leq C_\ve \sqrt{T} \(\|f^{in}\|^4_{L^\infty(\RR^d;H^s(\BS))}+\|g\|^4_{\mathscr Y_s}\).
 \end{split}
 \end{equation}
 If we denote $B_R$ to be the ball of radius $R$ in space $ \mathscr{Y}_s$, then by choosing $R \geq  \|f^{in}\|_{L^\infty(\RR^d;H^s(\BS))} $ and afterwards choosing $T\leq \frac1{4C_\ve^2 R^4}$, we obtain that $\mathscr{F}(B_R)\subset B_R$. A similar estimate on the difference $\mathscr{F}g_1-\mathscr{F}g_2$ implies that $\mathscr{F}$ is a contraction on $\mathscr{Y}_s$ provided that $T\ll \|f^{in}\|^{-4}_{L^\infty(\RR^d;H^s(\BS))}$. So $\mathscr{F}$ must have a unique  fixed point  and this leads to the local in time solution of \eqref{main:eqn}.

To extend the solution to be a unique global in time one, it follows from   \eqref{eq:1.61} that, the equation \eqref{main:eqn}  can be considered as a heat equation over $\BS$ with uniformly bounded coefficient
 \begin{equation}\label{eq:1.45}
   \ve\p_t f=\Delta_\BS f+\CR\CU_\ve[f]\cdot\CR f+ f \Delta_\BS \CU_\ve [f].
 \end{equation}
So the standard energy estimate implies the existence and uniqueness of solution on $[0,\infty)$.

\textit{Step 2: Regularity of the solution.}
In the previous step, we show  $f\in \mathscr{X}_s$, defined by \eqref{last1}.  So for every $T>0$,  there exists at least one $\tau\in [0,T)$ such that $f\mid_{t=\tau}\in L^\infty(\RR^d;H^{s+1}(\BS))$. Using this as initial data and solve \eqref{main:eqn} on $[\tau,T)$, the previous step, especially the uniqueness, implies
$$(\p_t f,\Delta_\BS f)\in L^\infty(\RR^d;C([\tau,\infty);H^{s+1}(\BS))).$$
Since this argument applies to every $T>0$, we conclude that
$$(\p_t f,\Delta_\BS f)\in L^\infty(\RR^d;C((0,\infty);H^{s+1}(\BS)))$$
 and thus more spatial regularity in $m\in\BS$ can be deduced  if we repeat this argument. Finally we obtain the instantaneous regularity
\begin{equation}\label{last2}
f\in L^\infty(\RR^d;C^\infty((0,\infty)\times \BS ))\cap \mathscr{X}_s.
\end{equation}

\textit{Step 3: Positivity of  the solution.}
We first prove the positivity of solution by assuming that $f^{in}\in L^\infty(\RR^d;H^s(\BS)\cap C(\BS))$. With the additional assumption on the continuity of $f^{in}$ in $\BS$,   it follows from \eqref{last2} that, for sufficiently small time $0<\tau\ll 1$, we have $f>\delta/2$ on $[0,\tau)$ and then $f$ becomes smooth in $[\tau,\infty)\times\BS$. So we can write (\ref{eq:1.45}) as
\begin{equation*}
  \ve\p_t f=\Delta_\BS f+\CR\CU_\ve[f]\cdot\CR f+ f G,
\end{equation*}
where
$$G(t, m, x)=\Delta_{\BS}\CU_\ve[f]={6\alpha(m\otimes m-\tfrac13\mathbb{ I}_3):  Q[f]* k_\ve.}$$
For almost every $x\in\mathbb{R}^d$, we denote by $T_x>0$   the first time such that
\begin{equation}\label{eq:1.65}
\inf_{m\in\BS}f(T_x, x, m)=0.
\end{equation}
 Then for every $t\in[0, T_x)$, it holds $f>0$ and we consider $\tilde f(t, x, m)=f e^{\frac 6\ve\int_0^t|Q[f]|}$. It can be readily verified that
\begin{equation*}
\begin{split}
  &\ve\p_t \tilde{f}=e^{\frac 6\ve\int_0^t|Q[f]|}(\ve\partial_tf+6|Q[f]| f)\\
  \ge & e^{\frac 6\ve\int_0^t|Q[f_\ve]|}( \Delta_\BS f+\CR\CU_\ve[f]\cdot\CR f)= \Delta_\BS\tilde f+\CR\CU_\ve[f]\cdot\CR \tilde f.
\end{split}
\end{equation*}
So the weak maximum principle implies that $\tilde f(m,x,t)$ attains its minimum on $\{0\}\times\BS$ for fixed $x$, that is
\begin{equation}\label{last5}
f(t, x, m)\ge \inf_{m\in\BS}f^{in}(x, m)e^{-\frac 6\ve\int_0^t|Q[f]|} >0,\qquad \text{for } t\le T_x, m\in\BS,
\end{equation}
which contradicts \eqref{eq:1.65}. Thus $f$ stays positive and the above estimate is valid for every $t\geq 0$. Moreover, \eqref{last5} gives  the lower bound for the decay in \eqref{last3} and it is easy to obtain that $\int_\BS f(m,x,t)\ud m=1$ according to \eqref{eq:1.19}.

If we abandon the assumption on the continuity of $f^{in}$ in $m\in\BS$, that is assume we have $f^{in}\in L^\infty(\RR^d;H^s(\BS))$, then we can find a family of approximation $f^{in}_{(n)}$, indexed by $n\in\mathbb{N}^*$, such that $f^{in}_{(n)}\geq \delta/2~a.e.~\text{in}~\RR^d\times\BS,~f^{in}_{(n)}\in L^\infty(\RR^d;H^s(\BS)\cap C(\BS)) $ such that
$$  f^{in}_{(n)}\xrightarrow{n\to \infty} f^{in}~\text{strongly in}~L^\infty(\RR^d;H^s(\BS)).$$
In view of \eqref{eq:1.61}, we can perform standard energy estimate, to show that the solution of \eqref{main:eqn} $f_{(n)}$ with initial data  $f^{in}_{(n)}$  is a Cauchy sequence in $\mathscr{X}_s$:
$$\|f_{(n)}-f_{(m)}\|_{\mathscr{X}_s}\leq C_\ve(T,f^{in}) \|f^{in}_{(n)}-f^{in}_{(m)}\|_{L^\infty(\RR^d;H^s(\BS))}.$$
 So $f_{(n)}\xrightarrow{n\to\infty} f\in\mathscr{X}_s$ and one can verify that $f$ solves \eqref{main:eqn} with initial data $f^{in}$ and is positive for almost every $x\in\RR^d$.
Therefore, we complete the proof of existence, uniqueness and instantaneous regularity   of positive solution $f$ with \eqref{last3}.

{\bf Part 2: Energy dissipation law.}
This part is devoted to the proof of \eqref{push1}. The main difficulty is brought by the lack of  integrability of $f$ and $Q[f]$ at $x=\infty$.

\textit{Step 1: Decay to constant distribution at $x=\infty$.}
The goal of this step is to prove the following estimate
\begin{equation}\label{eq:1.66}
  \|f(\cdot,t)-f_{e_0}\|_{L^2( \RR^d\times\BS)}\leq  e^{Ct}\|f^{in}- f_{e_0} \|_{L^2( \RR^d \times \BS )}.
\end{equation}
First of all, we make the assertion that  $f_{e_0}=\frac 1Z e^{\eta(m\cdot e_0)^2}$ is a solution to \eqref{eq:1.45} for fixed $e_0\in\BS$. Actually, since $f_{e_0}$ is $x$-independent, we have
\begin{equation*}
 \CU_\ve[f_{e_0}]=\CU_0[f_{e_0}],\quad   \mathcal{L}_\ve Q[f_{e_0}]=0
\end{equation*}
according to Lemma \ref{uniaxial1} and  formula \eqref{def:L}. Moreover,
\begin{equation}\label{eq:1.33}
  \ve\p_t f_{e_0}-\CR(f_{e_0}\CR (\log f_{e_0}+\CU_\ve[f_{e_0}]))= -\CR(f_{e_0}\CR (\log f_{e_0}+\CU_0[f_{e_0}])).
\end{equation}
On the other hand, since $f_{e_0}$ is the global minimizer of the homogeneous Maier-Saupe energy, according to Proposition \ref{prop:critical}, we have
\begin{equation*}
  \log f_{e_0}+\CU_0[f_{e_0}]\equiv const,
\end{equation*}
and together with \eqref{eq:1.33}
\begin{equation}\label{eq:1.04}
  \ve\p_t f_{e_0}=\Delta_\BS f_{e_0}+\CR f_{e_0}\cdot \CR \CU_\ve[f_{e_0}]+f_{e_0}\Delta_\BS \CU_\ve[f_{e_0}].
\end{equation}

 Now we rewrite \eqref{eq:1.45} in the similar form of \eqref{eq:1.04}:
\[\ve\p_t f=\Delta_\BS f+\CR f\cdot \CR \CU_\ve[f]+f\Delta_\BS \CU_\ve[f].\]
Subtracting \eqref{eq:1.04} by \eqref{eq:1.33} leads to the equation for $g:=f-f_{e_0}$,
\[\ve\p_t g-\Delta_\BS g= \CR g\cdot\CR\CU_\ve[f]+\CR f_{e_0}\cdot \CR \CU_\ve[g]+g\Delta_\BS \CU_\ve[f]+f_{e_0}\Delta_\BS \CU_\ve [g].\]
In view of  \eqref{eq:RU}, for almost every $x\in\RR^d$, the above equation is  a homogenous linear parabolic equation on $\BS$ with  uniformly bounded coefficient (depending on $\ve$). So   it follows from standard energy method that
\[\ve\frac d{dt}\int_\BS g(m,\cdot)^2\ud m+\int_\BS |\CR g(m,\cdot)|^2\ud m\leq C_\ve\int_\BS g^2(m,\cdot)\ud m,~a.e.~(x,t)\in\RR^d\times\RR_+\]
 and thus
\[\|g(\cdot,t)\|^2_{L^2(\RR^d\times \BS)}\leq  e^{Ct}\|f^{in}(\cdot)-f_{e_0}(\cdot)\|^2_{L^2(\RR^d\times\BS)},\]
which yields \eqref{eq:1.66}.

\textit{Step 2: Energy dissipation law.} Define
\begin{align*}
  \tilde{Q}[f]:=Q[f]-Q[f_{e_0}],
\end{align*}
which belongs to $L^\infty(0,T;L^2(\RR^d))$ owning to \eqref{eq:1.66}. Thus, we have from Lemma \ref{converge-T} that
\begin{equation}\label{eq:1.35}
\mathcal{L}_\ve \tilde Q[f],~\mathcal{T}_\ve \tilde Q[f]\in {L^\infty(0,T;L^2(\RR^d))}.
\end{equation}
Now we show that
\begin{equation}\label{eq:1.67}
\p_t Q[f] \in {L^\infty(0,T;L^2(\RR^d))}.
\end{equation}
 To this end, we multiply \eqref{eq:1.45} by $m\otimes m -\tfrac 13 I_3$ and integrate over $\BS$. This gives
 \begin{equation}\label{eq:1.05}
   \p_t Q[f]=-6Q[f ]+2\alpha \CM_{f }(Q[f ]*k_\ve)=-6Q[f ]+2\alpha \CM_{{f }}(Q[f ])-2\ve \alpha \CM_{f }(\mathcal{L}_\ve \tilde Q[f ] ),
 \end{equation}
 where $\mathcal{M}_{f}$ is a linear operator  defined, for any $3\times 3$ matrix $A$, by
\begin{equation*}
  \mathcal{M}_{f}(A)=\frac23A+Q[f]\cdot A+A\cdot Q[f]-2A:\int_{\BS}m^{\otimes 4} f(\cdot, m) \ud m.
\end{equation*}
The first equality in \eqref{eq:1.05}  will be derived in Remark \ref{last4} below and the second one is a consequence of \eqref{def:L} and the linearity of $\mathcal{M}_f$.
 As $f=f_{e_0}$ is an equilibrium solution of  \eqref{eq:1.04},  $Q[f_{e_0}]$ is an
equilibrium solution of \eqref{eq:1.05}. This together with   $\mathcal{L}_\ve Q[f_{e_0}]=0$ leads to
 $$-6Q[f_{e_0}]+2\alpha \CM_{f_{e_0}}(Q[f_{e_0}])=0.$$
In view of  \eqref{eq:1.66}, we arrive at
$$-6Q[f ]+2\alpha \CM_{f }(Q[f ])\in {L^\infty(0,T; L^2(\RR^d))}$$
and the proof  of \eqref{eq:1.67} is achieved.

To establish \eqref{push1},
we choose  a cut-off function $\phi\in C_c^1(\RR^d)$ such  that  $\phi(x)=1$ for $|x|\leq 1$ and define $\phi_R(x)=\phi(x/R)$. Then, it follows from \eqref{relation:U} and  \eqref{last3} that
\begin{equation*}
\begin{split}
&-\frac 1\ve\int_{\RR^d\times\BS} f\big|\CR\mu_\ve[f]\big|^2 \phi_R\\
=&\int_{\RR^d\times\BS} f_t(\log f+\CU_\ve[f]) \phi_R\\
=&\int_{\RR^d\times\BS}\(\frac d{dt} f\log f \phi_R+f_t \CU_{0}[f]\phi_R+f_t\ve\alpha(m\otimes m):\mathcal{L}_\ve Q[f]\phi_R\)\ud m\ud x\\
=&\frac d{dt} \int_{\RR^d} (\mathcal{E}_0[f]-E_0) \phi_R +\ve\alpha\int_{\RR^d}\p_t Q[f]:\mathcal{L}_\ve Q[f]\phi_R\\
=&\frac d{dt} \int_{\RR^d} (\mathcal{E}_0[f]-E_0) \phi_R +\ve\alpha\int_{\RR^d}\p_t \tilde Q[f]:\mathcal{L}_\ve \tilde Q[f]\phi_R\\
=&\frac d{dt} \int_{\RR^d} (\mathcal{E}_0[f]-E_0)\phi_R +\frac{\ve\alpha}2\frac d {dt}\int_{\RR^d}\phi_R|\mathcal{T}_\ve  \tilde Q[f]|^2-\alpha\ve  \int_{\RR^d} \p_t  \tilde Q[f]:   [\mathcal{T}_\ve,\phi_R] \cdot\mathcal{T}_\ve  \tilde Q[f].
\end{split}
\end{equation*}

Integrating the above identity in $t$ leads to the localized energy dissipation law:
\begin{equation}\label{last6}
\begin{split}
 &  \int_{\RR^d} \(\mathcal{E}_0[f(\cdot,t)]-{E}_0+\frac{\ve\alpha}2|\mathcal{T}_\ve  \tilde Q[f(\cdot,t)]|^2\) \phi_R +\frac 1\ve\int_0^t\int_{\RR^d\times\BS} f\big|\CR\mu_\ve[f]\big|^2 \phi_R\\
 &{-\alpha\ve  \int_{\RR^d} \p_t  \tilde Q[f]:   [\mathcal{T}_\ve,\phi_R] \cdot\mathcal{T}_\ve  \tilde Q[f]}= \int_{\RR^d} \(\mathcal{E}_0[f^{in}]-{E}_0+\frac{\ve\alpha}2|\mathcal{T}_\ve  \tilde Q[f^{in}]|^2\) \phi_R\\
\end{split}
\end{equation}
Now we claim that
\[{ \int_{\RR^d} \p_t  \tilde Q[f]:   [\mathcal{T}_\ve,\phi_R] \cdot\mathcal{T}_\ve  \tilde Q[f]\ud x \xrightarrow{R\to\infty}0},~a.e.~\text{on}~(0,t).\]
Actually, owning to \eqref{eq:1.35} and \eqref{eq:1.67}, we only need to show that
\begin{equation}\label{eq:1.36}
  [\mathcal{T}_\ve,\phi_R] g\xrightarrow{R\to\infty}0,~\forall g\in L^2(\RR^d).
\end{equation}
To show this, noticing that the $\mathcal{T}_\ve$ is bounded in $L^2(\RR^d)$ (see Lemma \ref{converge-T})
\begin{align*}
    \|[\mathcal{T}_\ve,\phi_R] g\|_{L^2(\RR^d)}&\leq\|\mathcal{T}_\ve((\phi_R-1)g)\|_{L^2(\RR^d)}+\|(1-\phi_R)\mathcal{T}_\ve g\|_{L^2(\RR^d)}\\
    &\leq C_\ve\(\| (\phi_R-1)g\|_{L^2(\RR^d)}+\|(1-\phi_R)\mathcal{T}_\ve g\|_{L^2(\RR^d)}\).
\end{align*}
Then applying dominated convergence theorem to the last two components leads to \eqref{eq:1.36} and thus the claim has been justified.
Notice also that all the rest terms in \eqref{last6} are non-negative and non-decreasing in $R$.
So sending $R\to \infty$ in \eqref{last6} leads to
\begin{equation*}
\begin{split}
 &  \int_{\RR^d} \(\mathcal{E}_0[f(\cdot,t)]-E_0+\frac{\ve\alpha}2|\mathcal{T}_\ve  \tilde Q[f(\cdot,t)]|^2\)  +\frac 1\ve\int_0^t\int_{\RR^d\times\BS} f\big|\CR\mu_\ve[f]\big|^2\\
 =& \int_{\RR^d} \(\mathcal{E}_0[f^{in}]-E_0+\frac{\ve\alpha}2|\mathcal{T}_\ve  \tilde Q[f^{in}]|^2\).
\end{split}
\end{equation*}
Then using \eqref{eq:1.68} and the fact that $$\int_{\RR^d}|\mathcal{T}_\ve  \tilde Q[f(\cdot,t)]|^2= \int_{\RR^d} \tilde Q[f]:\mathcal{L}_\ve \tilde Q[f]= \int_{\RR^d}  Q[f]:\mathcal{L}_\ve  Q[f],$$
we obtain \eqref{push1} as well as \eqref{eq:ME}.
\end{proof}

\begin{Remark}\label{last4}
For completeness, we give  the derivation of \eqref{eq:1.05} by calculating the second moment of the right hand side of \eqref{eq:1.45}. For every constant symmetric matrix $D=\{D_{ij}\}_{1\leq i,j\leq 3}$:
\begin{align*}
   & \int_\BS \Big( \Delta_\BS f+\CR\cdot(f\CR\CU_\ve[f])\Big)(m_im_j-\tfrac13\delta_{ij}) D_{ij} dm\\
=&\int_\BS \Big(  f\Delta_{\BS}(m_im_j-\tfrac13\delta_{ij}) D_{ij} -f\CR\CU_\ve[f]\cdot\CR\big(m_im_j D_{ij}\big)\Big) dm\\
=&  \int_\BS \Big( -6 f (m_im_j-\tfrac13\delta_{ij}) D_{ij} +4\alpha f  m\wedge\big( (Q[f]*k_\ve)\cdot m\big)\cdot\big(m\wedge(D\cdot m)\big)\Big) \ud m \\
=&-6 Q_{ij}[f]:D_{ij}+4\alpha Q_{ij}[f]* k_\ve Q_{j\ell}[f] D_{i\ell}+\tfrac{4\alpha}3 Q_{i\ell}[f]* k_\ve D_{i\ell}-4\alpha D:\int_\BS m^{\otimes 4} f\ud m:(Q[f]*k_\ve),
\end{align*}
where   we employed \eqref{identity:R-1}, \eqref{Appd:eq:R3},  \eqref{eq:RU}, \eqref{Appd:eq:R2}
and the following Cauchy-Binet identity successively
\[(m\wedge u)\cdot(m\wedge v)=u\cdot v-(m\cdot u)(m\cdot v),~\quad\text{for } |m|=1.\]
The above formula together with  $\mathcal{L}_\ve Q[f_{e_0}]=0$ implies the first equality in \eqref{eq:1.05} since $D_{ij}$ is any symmetric matrix.
We note that, by closing the fourth-order moment utilizing the Bingham closure, \eqref{eq:1.05} can be used to derive a closed $Q$-tensor system, see \cite{HLWZZ} for details.
\end{Remark}

In the sequel, to figure out the dependence on $\ve$, we use $f_\ve$ to denote the solutions to \eqref{main:eqn} constructed  in Theorem \ref{thm:wellpose}. Since $f_\ve(m,\cdot)$ is a family of probability density,
\begin{align}\label{bound-Q}
\|Q[f_\ve]\|_{L^\infty(\mathbb{R}^d)}\le \frac23,\quad \|Q[f_\ve]* k_\ve\|_{L^\infty(\mathbb{R}^d)}\le \frac23.
\end{align}
Therefore, we infer  from (\ref{eq:UQ}), (\ref{eq:RU}) and (\ref{eq:R2U}) that
\begin{align}\label{bound-U}
\big\|(\CU_\ve[f_\ve], \CR\CU_\ve[f_\ve],\Delta_\BS \CU_\ve[f_\ve])\big\|_{L^\infty(\mathbb{R}^d)}\le C.
\end{align}
Note that here and in the sequel, $C$ will be a generic positive constant which might change from line to line and will be  independent of $\ve$.
\begin{Proposition}
Under the assumptions of Theorem \ref{thm:converge}.
Let $f_\ve$ be solutions to Doi-Onsager \eqref{main:eqn}. Then, for every  $T>0$ and {every $\delta\in (0,T)$},
\begin{align}\label{eq:1.71}
&{\|\CR f_\ve\|_{L^\infty(\RR^d;L^2(\BS\times (0,T)))}\leq C},\\
\label{eq:1.29}
&{\|\Delta_\BS f_\ve\|_{L^\infty(\RR^d ;L^2(\BS\times (\delta,T)))}\leq C\delta^{-1}},\\
\label{eq:1.17}
&\|\p_t (Q[f_\ve]*k_\ve)\|_{L^2(\RR^d\times (0,T))} \leq C,\\
&\| Q[f_\ve]-Q[f_{e_0}]\|_{L^\infty(0,T;L^2(\RR^d))} \leq C,\label{eq:1.21}
\end{align}
where $C$ is a constant independent of $\ve$.
\end{Proposition}
\begin{proof}
First, we prove
\begin{equation}\label{eq:1.16}
{\ve\| f_\ve\|^2_{L^\infty(\RR^d\times (0,T);L^2(\BS))}+  \|\CR f_\ve\|^2_{L^\infty(\RR^d;L^2(\BS\times (0,T)))}\leq C\ve\|f_{\ve}^{in}\|_{L^\infty(\RR^d;L^2(\BS))}^2+CT}.
\end{equation}
To this end, we test the   equation  \eqref{main:eqn} by $f_\ve$ and integrate by parts over $\BS$:
\begin{equation*}
  \begin{split}
  \ve\frac \ud{\ud t}\int_\BS f^2_\ve+  \int_\BS |\CR f_\ve|^2\ud m=&\int_\BS f_\ve\CR\CU_\ve[f_\ve]\cdot\CR f_\ve \ud m\\
  =&-\frac12\int_\BS \Delta_\BS\CU_\ve[f_\ve] f_\ve^2 \ud m \leq C\int_\BS  f^2_\ve\ud m.
  \end{split}
\end{equation*}
In the last step, we employed \eqref{bound-U}.
On the other hand, it follows from $\CR\cdot\CR=\Delta_\BS$ and the Nash inequality in  \cite{MR1704184}   that
\begin{equation}
{\|\varphi\|_{L^2(\BS)}^2\le C\|\sqrt{-\Delta_\BS}\varphi\|_{L^2(\BS)}\|\varphi\|_{L^1(\BS)}+C\|\varphi\|_{L^1(\BS)}^2},~\forall \varphi\in C^1(\BS).\label{eq:inter}
\end{equation}
Applying to $f_\ve$ leads to
\begin{equation*}
{\|f_\ve\|_{L^2(\BS)}^2\le C\|\sqrt{-\Delta_\BS}f_\ve\|_{L^2(\BS)}\|f_\ve\|_{L^1(\BS)}+C\|f_\ve\|_{L^1(\BS)}^2\leq C\(1+\|\CR f_\ve\|_{L^2(\BS)} \)}.
\end{equation*}
Combining the previous two inequalities, we arrive at
\begin{equation*}
\ve\frac \ud{\ud t}\int_\BS f^2_\ve\ud m+\frac 12\int_\BS |\CR f_\ve|^2\ud m\leq C.
\end{equation*}
 Integrating the above inequality in $t$   implies \eqref{eq:1.16}.
In order to obtain the higher order estimate \eqref{eq:1.29}, we rewrite \eqref{main:eqn} as
\begin{equation}\label{eq:1.28}
  \ve\pa_t{f_\ve}-\Delta_\BS f_\ve= \CR\cdot( f_\ve\CR\CU_\ve[f_\ve])=:g_\ve.
\end{equation}
then using \eqref{bound-U}
 \begin{equation*}
|g_\ve|    \leq C\(|\CR f_\ve|{+| f_\ve|}\)~a.e. ~(m,x,t)\in \BS\times\RR^d\times (0,T)
\end{equation*}
and thus
\[\|g_\ve\|^2_{L^\infty(\RR^d;L^2(\BS\times (0,T)))}\leq C\ve\|f_{\ve}^{in}\|_{L^\infty(\RR^d;L^2(\BS))}^2 .\]
Now we multiply \eqref{eq:1.28}  by $t\Delta_\BS f_\ve$ and follow the standard energy estimate:
\begin{align*}
&\frac 12\frac d{dt}\int_\BS t|\CR f_\ve|^2-\frac 12\int_\BS |\CR f_\ve|^2+\int_\BS t|\Delta_\BS f_\ve|^2\\
&\qquad=\int_\BS \sqrt{t}\Delta_\BS f_\ve\sqrt{t} g_\ve\leq \frac 12\int_\BS t|\Delta_\BS  f_\ve|^2+\frac 12\int_\BS t g_\ve^2.
\end{align*}
The  above two  estimates together lead  to \eqref{eq:1.29}.

To derive \eqref{eq:1.17}, we test \eqref{main:eqn} by any $\psi(m)\in C^\infty(\BS)$ and integrate by parts over $\BS$
\begin{equation}\label{eq:1.20}
  \begin{split}
     \p_t\int_\BS  f_\ve(m,x,t)\psi(m) \ud m &=-\frac 1\ve\int_\BS f_\ve\CR\mu_\ve[f_\ve]\cdot\CR\psi \ud m\\
     &=-\frac 1\ve\int_\BS \sqrt{ f_\ve}\CR\mu_\ve[f_\ve]\cdot \sqrt{f_\ve}\CR\psi\ud m.
  \end{split}
\end{equation}
    Applying the Cauchy-Schwarz inequality, we get
    \begin{equation*}
 \left|\p_t\int_\BS  f_\ve(m,x,t)\psi(m)\ud m \right|^2 \leq \frac 1{\ve^2}\int_\BS f_\ve|\CR\mu_\ve[f_\ve]|^2\ud m\int_\BS f_\ve|\CR\psi|^2\ud m.
    \end{equation*}
In particular, if we take $$\psi(m)=m_im_j-\tfrac 13\delta_{ij}~(1\leq i,j\leq 3)$$
in the above inequality and combine it with \eqref{push1}, then we arrive at
    \begin{equation*}
    \|\p_t (Q[f_\ve]*k_\ve)\|_{L^2(\RR^d\times (0,T))}= \| \p_t Q[f_\ve]*k_\ve\|_{L^2(\RR^d\times (0,T))}\leq   \|\p_t Q[f_\ve]\|_{L^2(\RR^d\times (0,T))}\leq C,
    \end{equation*}
which yields  \eqref{eq:1.17}.

To prove \eqref{eq:1.21}, we use \eqref{eq:1.20} again to get
\begin{equation*}
   \p_t Q[  f_\ve](x,t)-\p_t Q[f_{e_0}]=-\frac 1\ve\int_\BS \sqrt{ f_\ve}\CR\mu_\ve[f_\ve]\cdot \sqrt{f_\ve}\CR (m\otimes m-\tfrac 13I_3)\ud m.
\end{equation*}
Testing by  $Q[  f_\ve]-Q[f_{e_0}]$ and performing standard energy estimates leads to
\begin{equation}\label{eq:1.22}
  E'(t)\leq CA(t) \sqrt{E(t)}
\end{equation}
where \[E(t)=\|Q[  f_\ve]-Q[f_{e_0}]\|^2_{L^2(\RR^d)},~A^2(t)=\frac 1{\ve^2}\int_{\RR^d\times\BS} f_\ve|\CR\mu_\ve[f_\ve]|^2\ud m\ud x.\]
Solving differential inequality \eqref{eq:1.22} together with initial condition \eqref{assump:uniform} leads to \eqref{eq:1.21}.
\end{proof}

\section{Compactness of the second moments}

In this section, we study the compactness  and convergence of the second moments $Q[f_\ve]$ via the relative-energy estimate \eqref{push1}.
\begin{Proposition}\label{compactness1}
Modulo the extraction of a subsequence, it holds that for any $T>0$,
\begin{equation}\label{20150626claim1}
 Q[f_\ve]\xrightarrow{\ve\to 0} \Psi\quad \text{strongly in }~C([0,T];L_{loc}^2(\RR^{d})),
\end{equation}
for some  $\Psi\in L^\infty(0,T; H^1_{loc}(\RR^d))$. Moreover,
\begin{equation}\label{20150626claim2}
\nabla (Q[f_\ve]*k_\ve)\xrightarrow{\ve\to 0}\nabla \Psi\quad\text{weakly-{star} in}~L^\infty(0,T;L^2 (\RR^d)).
\end{equation}
\end{Proposition}
\begin{proof}
The assertion \eqref{20150626claim1} is a consequence of the following estimate
\begin{equation}\label{20150626bound1}
\sup_{0\leq t\leq T}\frac 1\ve\int_{\RR^d} |Q[f_\ve]*k_\ve-Q[f_\ve]|^2 \ud x+\sup_{0\leq t\leq T}\int_{\RR^d} |\nabla (Q[f_\ve]*k_\ve)|^2 \ud x\leq  C.
\end{equation}
Actually, it follows from \eqref{20150626bound1}, \eqref{eq:1.17} and the Aubin-Lions lemma (see for instance \cite{Simon1987}) that, up to the extraction of a subsequence, $\{Q[f_\ve]*k_\ve\}_{\ve>0}$ is compact in $C([0,T];L_{loc}^2(\RR^d))$ and this together with the following inequality implies the strong convergence of $u_\ve:=Q[f_\ve](x,t)$ in $L^\infty(0,T;L_{loc}^2 (\RR^d))$:
\begin{equation*}
 |u_\ve-u_\sigma|\leq |u_\ve-u_\ve*k_\ve|+|u_\sigma-u_\sigma*k_\sigma|+|u_\sigma *k_\sigma-u_\ve*k_\ve|.
\end{equation*}

For the assertion \eqref{20150626claim2}, we have
\begin{equation}
  \nabla (u_\ve* k_\ve)\xrightarrow{\ve\to 0} \Phi=\{\Phi_j\}_{1\leq j\leq d}~\text{weakly-{star} in}~L^\infty(0,T;L^2 (\RR^{d})).\nonumber
\end{equation}
On the other hand, for any $\varphi(x,t)\in C_c^\infty(\RR^d\times (0,T);\RR^{3\times 3})$,
\begin{equation*}
 \begin{split}
    -\int_0^T\int_{\RR^d}\p_j (u_\ve* k_\ve):\varphi \ud x\ud t&=\int_0^T\int_{\RR^d} (u_\ve* k_\ve):\p_j\varphi\ud x\ud t\\
    &=\int_0^T\int_{\RR^d} u_\ve :(k_\ve*\p_j\varphi)\ud x\ud t.
  \end{split}
\end{equation*}
Taking $\ve\to 0$ leads to
\begin{equation*}
  -\int_0^T\int_{\RR^d}\varphi :\Phi_j\ud x\ud t =\int_0^T\int_{\RR^d} \Psi:\p_j\varphi\ud x\ud t,
\end{equation*}
which implies
\begin{equation*}
\nabla \Psi=\Phi\in L^\infty(0,T;L^2(\RR^d))
\end{equation*}
and \eqref{20150626claim2} follows.

The proof of  \eqref{20150626bound1} was   motivated by \cite{AlbertiBellettini1998a,LiuWang2016}.  First, we infer from the assumption \eqref{assump1} for the kernel function $k(x)$ that
\begin{equation}\label{bound1}
\begin{split}
&\int_{\RR^d} |u_\ve*k_\ve-u_\ve|^2 \ud x
=\int_{\RR^d}\left|\(1-\hat{k} (\sqrt{\ve}\xi)\)\hat{u}_\ve(\xi)\right|^2\ud \xi\\
&\le  2\int_{\RR^d}\left|\sqrt{1-\hat{k} (\sqrt{\ve}\xi)}\hat{u}_\ve(\xi)\right|^2\ud \xi
=2\int_{\RR^d}\hat{u}_\ve(\xi):\hat{u}_\ve(\xi)- \hat{k} (\sqrt{\ve}\xi) \hat{u}_\ve(\xi):\hat{u}_\ve(\xi)\ud \xi\\
&=\int_{\RR^d\times\RR^d} k_\ve(x-y)|u_\ve(x)-u_\ve(y)|^2\ud x\ud y.
\end{split}
\end{equation}
Similarly, we infer from (\ref{assump1}) that
\begin{multline}\label{bound2}
  \int_{\RR^d} |\nabla (k_\ve*u_\ve)|^2 \ud x=4\pi^2\int_{\RR^d}\left|\xi\hat{k}(\sqrt{\ve}\xi)\hat{u}_\ve(\xi)\right|^2\ud \xi\\
  \leq \frac C\ve \int_{\RR^d}\left|\sqrt{(1-\hat{k} (\sqrt{\ve}\xi))}\hat{u}_\ve(\xi)\right|^2 \ud\xi
      =\frac C\ve   \int_{\RR^d\times\RR^d}  |u_\ve(x)-u_\ve(y)|^2k_\ve(x-y) \ud x \ud y.
\end{multline}
Then we can combine \eqref{bound1}-\eqref{bound2} with \eqref{push1} to get \eqref{20150626bound1}.
\end{proof}

The following proposition gives the characterization of the limit function $\Psi$ in Proposition \ref{compactness1}.
\begin{Proposition}\label{prop:compact2}
For any $T>0$ and compact set $W\subseteq \RR^d$, modulo the extraction of a subsequence,   it holds that   $f_\ve$ is uniformly integrable on $\BS\times W\times (0,T)$ and
  \begin{equation*}
    {f_\ve}\xrightarrow{\ve\to 0} {f}_0\quad \text{weakly in}\,\,L^1\big(\BS\times W\times (0,T)\big),
  \end{equation*}
  where  ${f}_0(m,x,t)$ is given by
  \begin{equation*}
    {f}_0(m,x,t)=h_{n(x,t)}(m):=\frac1Z e^{\eta(m\cdot n(x,t))^2}
  \end{equation*}
for some unit vector field $n:(0,T)\times\RR^d\mapsto \BS$ such that
 \begin{equation}\label{eq:1.27}
{n-e_0\in L^\infty(0,T;H^1(\RR^d)),~n_t\in L^2((0,T)\times\RR^d).}
 \end{equation}
 In addition, we have $$\Psi=Q[{f_0}]=S_2(n\otimes n-\tfrac 13I_3)~ a.e.~ in ~\RR^d\times (0,T)$$ where $S_2\neq 0$ is defined at \eqref{eq:1.58} and
\begin{equation}\label{limit:Qf}
 Q[f_\ve]\xrightarrow{\ve\to 0} Q[{f_0}]=\Psi\quad \text{strongly in }~C([0,T];L_{loc}^2(\RR^{d})).
\end{equation}
\end{Proposition}
\begin{proof}
First of all, we show that
  \begin{equation}\label{extension6}
     {f_\ve}\xrightarrow{\ve\to 0} {f}_0\quad \text{weakly in}\,\,L^1\big(\BS\times W\times (0,T)\big),
  \end{equation}
  for some local equilibrium distribution $ {f_0}(m,x,t)$. Indeed, we   deduce from  \eqref{push1} that
\begin{equation}
{\sup_{t\in (0,T)}\int_{\RR^d} (\mathcal{E}_0[f_\ve](x,t)-{E}_0 )\ud x \le C \ve},\label{eq:E0-conv}
\end{equation}
and thus for any compact set $W\subset\RR^d$,
\begin{equation*}
  \sup_{t\in (0,T)}\int_{W} \mathcal{E}_0[f_\ve](x,t) \ud x \leq C|W|+C\ve.
\end{equation*}
Thanks to \eqref{energy:MS1} and the uniform bound (\ref{bound-Q}) for $ |Q[f_\ve](x,t)|$, we obtain the entropy estimate
\begin{equation*}
 \sup_{t\in (0,T)} \int_{W\times\BS} {f_\ve}\ln  {f_\ve}\ud x\ud m\leq   C(|W|+1).
\end{equation*}
Then Lemma \ref{lem:convex} leads to the uniformly integrability of $\{f_\ve\}_{\ve>0}$ and \eqref{extension6}.

To show that $ {f_0}$ is a local equilibrium distribution, we deduce from \eqref{eq:E0-conv} and the fact that $f_{e_0}$ is a global minimizer of $ \mathcal{E}_0$ (by Lemma \ref{lem:minimizer}) that
\begin{equation*}
0\leq \sup_{t\in (0,T)} \lim_{\ve\to 0}\int_{\RR^d}\Big(\mathcal{E}_0[ {f_\ve}](x,t)-{E}_0]\Big)\ud x= 0.
\end{equation*}
{In view of \eqref{energy:MS1}, Lemma \ref{lem:convex} and strong compactness of $Q[f_\ve]$ \eqref{20150626claim1}, we can exchange  the limit $\ve\to 0$  and the integral in the above inequality and get}
\begin{equation*}
\mathcal{E}_0[f_0(\cdot,x, t)]= E_0,\quad\text{ for a.e. } (x,t)\in \RR^d\times (0,T).
\end{equation*}
Then Lemma \ref{lem:minimizer} ensures that there exists some function $n:\RR^d\times (0,T)\mapsto \BS$ such that
\begin{equation*}
   {f_0}(m,x,t)=\frac{e^{\eta(m\cdot n(x,t))^2}}{\int_\BS e^{\eta(m\cdot n(x,t))^2} \ud m} \quad  \text{a.e. } (x,t)\in\RR^d\times (0,T).
\end{equation*}
On the other hand, \eqref{extension6} imply that
\begin{equation*}
   Q[ {f_\ve}]\xrightarrow{\ve\to 0} Q[ {f_0}], \quad \text{weakly in }~L^1(W\times(0,T)).
 \end{equation*}
Together with \eqref{compactness1}, we obtain $\Psi=Q[ {f_0}]$ and \eqref{limit:Qf} follows.

 Consequently $f_0$ is a local equilibrium distribution whose $Q$-tensor belongs to $H_{loc}^1({\RR^d})$, for almost every $t\in [0,T]$. This together with   the orientability theorem in \cite[Theorem 2]{BallZarnescu2011} implies that $n(x,t)\in L^\infty\big(0,T;H_{loc}^1(\RR^d, \BS)\big).$
To show \eqref{eq:1.27}, it follows from \eqref{eq:1.21} and \eqref{limit:Qf} that, up to the extraction of a subsequence,
 \begin{equation*}
   Q[ {f_\ve}]-Q[f_{e_0}]\xrightarrow{\ve\to 0} Q[ {f_0}]-Q[f_{e_0}] \quad \text{weakly-{star} in }~L^\infty(0,T;L^2(\RR^{d })).
 \end{equation*}
 Then the weakly lower semicontinuity implies
 \[\|Q[ {f_0}]-Q[f_{e_0}]\|_{L^\infty(0,T;L^2(\RR^d))}\leq C.\]
 Since $f_0,f_{e_0}$ are both equilibrium solutions, we induce from Lemma \ref{uniaxial1} that
\begin{equation}\label{eq:1.62}
  \|n\otimes n-e_0\otimes e_0]\|_{L^\infty(0,T;L^2(\RR^d))}\leq C.
\end{equation}
 On the other hand,
  \eqref{limit:Qf} together with \eqref{20150626claim2} and \eqref{eq:1.17} implies
 \begin{equation*}
   \begin{split}
&      \nabla(Q[f_\ve]* k_\ve)\xrightarrow{\ve\to 0} \nabla \Psi, ~\text{weakly-{star} in}~L^\infty(0,T;L^2(\RR^d)),\\
&  \p_t(Q[f_\ve]* k_\ve)\xrightarrow{\ve\to 0} \p_t \Psi, ~\text{weakly in}~L^2((0,T)\times\RR^d).
   \end{split}
 \end{equation*}
  These together with $Q[f_0]=\Psi=S_2(n\otimes n-\tfrac 13 I_3)$ (from Lemma \ref{uniaxial1}) implies   \eqref{eq:1.27} except $n-e_0\in L^\infty(0,T;L^2(\RR^d))$. To complete the proof of \eqref{eq:1.27}, let \[\phi(t,x)=n(t,x)\cdot e_0\in[-1,1].\] Then we have from \eqref{eq:1.62} and the identity
   \[(n\otimes n-e_0\otimes e_0):(n\otimes n-e_0\otimes e_0)=2-2\phi^2\]
   that \[\nabla\phi\in L^\infty(0,T;L^2(\RR^d)),\quad 1-\phi^2\in L^\infty(0,T;L^1(\RR^d)).\]
By the following lemma, we have
\[1-\phi\in L^\infty(0,T;L^1(\RR^d)),~\text{or}~1+\phi\in L^\infty(0,T;L^1(\RR^d)).\]
This implies $n-e_0\in L^\infty(0,T;L^2(\RR^d))$ or $n+e_0\in L^\infty(0,T;L^2(\RR^d))$.
The second case can be  reduced to the first one if we replace $n$ by $-n$.  Thus the proof of the proposition is completed.
\end{proof}

 \begin{Lemma}
  Assume that $\phi:\RR^d\to[-1,1]$ satisfies  $\nabla\phi,  \sqrt{1-\phi^2}\in L^2(\RR^d)$. Then $1-\phi\in L^1(\RR^d)$ or $1+\phi\in L^1(\RR^d)$.
\end{Lemma}
\begin{proof}
Let $u(x)=\phi-\frac{\phi^3}3$ and $B_R=\{x\in\RR^d: |x|<R\}$. Obviously we have $u\in L^1(B_R)$. On the other hand it holds that
\begin{align*}
\int_{\RR^d}|\nabla u | \ud x \le \int_{\RR^d} |\nabla\phi|(1-\phi^2) \ud x \le \|\nabla \phi\|_{L^2(\RR^d)}\| 1-\phi^2 \|_{L^2(\RR^d)}.
\end{align*}
Thus $u\in BV(B_R)$. Let $E_t=\{x: u>t\}$ and $\|\partial E_t\|$ be the perimeter measure of $E_t$. Then it follows from the coarea-formula that
\begin{align*}
 \int_{-2/3}^{2/3} \|\partial E_t\|(B_R) \ud t  = \int_{B_R}|\nabla u| \ud x.
\end{align*}
Sending $R\to+\infty$, we have $\int_{-2/3}^{2/3} \|\partial E_t\|(\RR^d) \ud t <+\infty$.
Therefore, there exists $t\in (-\tfrac2 3,\tfrac 23)$ such that $  \|\partial E_t\|(\RR^d)<+\infty$ .
If we denote $|A|$ the Lebesgue measure of $A\subset\RR^d$, then it follows from the relative isoperimetric inequality \cite[Chapter 5]{Evans} that for any $R$
\begin{align*}
 \min\{ |E_t\cap B_R|,  |E_t^c\cap B_R|\}^{1-\frac 1d} \le C \|\partial E_t\|(B_R)\le C \|\partial E_t\|(\RR^d).
\end{align*}
Taking $R\to+\infty$ in the above inequalities leads to
\begin{align*}
|E_t|<+\infty, \text{ or } |E_t^c| <+\infty.
\end{align*}
Let $\delta$ be the unique number in $(-1, 1)$ such that $\delta-\delta^3/3=t$. Then
\begin{align*}
|\{\phi(x)<\delta\}|<+\infty, \text{ or } |\{\phi(x)\ge \delta\}| <+\infty.
\end{align*}
In the first case, we have
\begin{align*}
  \int_{\RR^d}(1-\phi) \ud x~ &=   \int_{\{\phi<\delta\}}(1-\phi) \ud x + \int_{\{\phi\ge \delta\}}(1-\phi) \ud x\\
  &\le 2 |\{\phi< \delta\}| +\frac{1}{1+\delta} \int_{\{\phi\ge \delta\}}(1-\phi^2) \ud x <+\infty.
\end{align*}
One can similarly obtain $  \int_{\RR^d}(1+\phi) \ud x<+\infty$ in  the other case and the lemma is proved.
\end{proof}

The following two lemmas are concerned with the properties of $\mathcal{T}_\ve$ and will be employed in the rest of the work.
Though the proof can be found in \cite{LiuWang2016} (except \eqref{eq:1.24}), we present them here for completeness.
\begin{Lemma}\label{lem:commu-l2}
For any $\varphi\in C_c^\infty(\RR^d)$, there exists a constant $C$ depending on $\varphi(x)$ but not on $\ve$ such that
\begin{equation}\label{convolu1}
  \|[\mathcal{T}_\ve, \varphi(x)] u \|_{L^2(\RR^d)}\leq C\|u\|_{L^2(\RR^d)}.
\end{equation}
\end{Lemma}
\begin{proof}
By the definition of the commutator, we have
 $$[\mathcal{T}_\ve,\varphi(x)] u=\mathcal{T}_\ve(\varphi(x)u(x))-\varphi(x)\mathcal{T}_\ve u(x).$$
Using Plancherel's theorem, Lemma \ref{lipschitz1} and and Young's inequality, we get that
\begin{equation*}
    \begin{split}
        &\|[\mathcal{T}_\ve,\varphi(x)] u\|_{L^2(\RR^d)}\\
        =&\frac 1{\sqrt{\ve}}\|  h(\sqrt\ve \xi)\hat{\varphi}*\hat{u}-\hat{\varphi}*( h(\sqrt\ve \xi) \hat{u}(\xi) )\|_{L^2(\RR^d)}\\
        =&\frac 1{\sqrt{\ve}}\left\|   h(\sqrt\ve \xi)\int_{\RR^d}\hat{\varphi}(\xi-\zeta) \hat{u}(\zeta)\ud \zeta-\int_{\RR^d}\hat{\varphi}(\xi-\zeta)  h(\sqrt\ve \zeta) \hat{u}(\zeta) \ud \zeta\right\|_{L^2(\RR^d)}\\
        =&\frac 1{\sqrt{\ve}}\left\| \int_{\RR^d}\hat{\varphi}(\xi-\zeta)(  h(\sqrt\ve \xi)-  h(\sqrt\ve \zeta)) \hat{u}(\zeta) \ud \zeta\right\|_{L^2(\RR^d)}\\
        \leq &\frac C{\sqrt{\ve}}\left\| \int_{\RR^d}\hat{\varphi}(\xi-\zeta)\sqrt{\ve} |\xi-\zeta|  \hat{u}(\zeta) \ud \zeta\right\|_{L^2(\RR^d)}\\
        = &C\left\|  (|\xi|\hat{\varphi}(\xi))* \hat{u} \right\|_{L^2(\RR^d)}\\
        \leq & C\||\xi|\hat{\varphi}(\xi)\|_{L^1(\RR^d)}\|\hat{u}\|_{L^2(\RR^d)}.
    \end{split}
  \end{equation*}
\end{proof}
\begin{Lemma}\label{variation5}
Up to the extraction of a subsequence, we have
\begin{equation*}
  \mathcal{T}_\ve Q[f_\ve]\xrightarrow {\ve\to 0} -i\sqrt{\tfrac \mu {2d}} \nabla Q[f_0],~\text{ weakly-{star} in }L^\infty(0,T;L^2(\RR^d)),
\end{equation*}
where $f_0$ is the limit of $f_\ve$ in Proposition \ref{prop:compact2}. Moreover, for any $\varphi(x)\in C_c^1(\RR^d)$, we have
\begin{equation}\label{eq:1.23}
  \mathcal{T}_\ve (\varphi Q[f_\ve])\xrightarrow {\ve\to 0} -i\sqrt{\tfrac \mu {2d}} \nabla (\varphi Q[f_0]),~\text{ weakly-{star} in }L^\infty(0,T;L^2(\RR^d)).
\end{equation}
\end{Lemma}

\begin{proof}
The uniform bound \eqref{push1} and the definition of $\mathcal{T}_\ve$ at \eqref{defteps} imply
\begin{equation}\label{eq:1.24}
\|\mathcal{T}_\ve Q[f_\ve ]\|_{L^\infty(0,T;L^2(\RR^d))} \leq C.
\end{equation}
Then there exists $\widetilde{Q}\in L^\infty(0,T;L^2(\RR^d))$ such that
\begin{equation*}
  \mathcal{T}_\ve Q[f_\ve ]\xrightarrow {\ve\to 0} \widetilde{Q}\quad\text{ weakly-{star} in } L^\infty(0,T;L^2(\RR^d)),
\end{equation*}
or equivalently, for every $\Phi\in C^\infty_c(\RR^d\times (0,T);\RR^{d}\times \RR^{ 3\times 3})$,
\begin{equation*}
\int_0^T\int_{\RR^d}\mathcal{T}_\ve Q[f_\ve]:\Phi(x,t)\ud x\ud t\xrightarrow{\ve\to 0}\int_0^T\int_{\RR^d}\widetilde{Q}(x,t):\Phi(x,t)\ud x\ud t.
\end{equation*}
On the other hand, the strong convergence of $Q[f_\ve ](x)$ stated in \eqref{20150626claim1} and Lemma \ref{converge-T} imply
\begin{align*}
\int_0^T\int_{\RR^d}\mathcal{T}_\ve Q[f_\ve]:\Phi(x,t)\ud x\ud t
&=-\int_0^T\int_{\RR^d}Q[f_\ve]: (\mathcal{T}_\ve\cdot\Phi(x,t))\ud x\ud t\\
&\xrightarrow {\ve\to 0}i\sqrt{\tfrac\mu {{2d}}}\int_0^T\int_{\RR^d}Q[f_0]:(\nabla\cdot \Phi(x,t))\ud x\ud t.
\end{align*}
The above two formulas together imply $\widetilde{Q}(x,t)=-i\sqrt{\tfrac \mu {{2d}}}\nabla Q[f_0]$.
Using the same method, we can show \eqref{eq:1.23} provided that
 \begin{equation}\label{eq:1.25}
   \|\mathcal{T}_\ve (\varphi Q[f_\ve])\|_{L^\infty(0,T;L^2(\RR^d))}\leq C,
 \end{equation}
 for some $C$ independent of $\ve$. Note that \eqref{eq:1.23} is not a straightforward consequence of \eqref{eq:1.24} as $\mathcal{T}_\ve$ is a non-local operator. To proceed, we write
\[\mathcal{T}_\ve(\varphi Q[f_\ve])=\varphi \mathcal{T}_\ve Q[f_\ve]+[\mathcal{T}_\ve,\varphi] Q[f_{e_0}]+[\mathcal{T}_\ve,\varphi](Q[f_\ve]-Q[f_{e_0}]).\]
The first two terms can be estimated by using (\ref{eq:1.24}) and the fact $[\mathcal{T}_\ve,\varphi] Q[f_{e_0}]=Q[f_{e_0}]\mathcal{T}_\ve\varphi$. For the last term,
we have from \eqref{convolu1} and \eqref{eq:1.21} that:
\[\|[\mathcal{T}_\ve,\varphi](Q[f_\ve]-Q[f_{e_0}])\|_{L^\infty(0,T;L^2(\RR^d))}\leq C\| (Q[f_\ve]-Q[f_{e_0}])\|_{L^\infty(0,T;L^2(\RR^d))}\leq C.\]
This implies \eqref{eq:1.25} and thus  \eqref{eq:1.23}.
\end{proof}

\section{Strong compactness via the dissipation control}
In this section and hereafter, we denote $f_0=f_0(m,x,t)$  the limiting equilibrium distribution function obtained in Proposition \ref{prop:compact2}, i. e.,
\begin{equation*}
   {f}_0(m,x,t)=h_{n(x,t)}(m).
\end{equation*}

\subsection{The linearized operator}
The linearized operator of $\CR\cdot\big(\CR f+f\CR \CU_0[f]\big)$ around $f_0$ is given by
 \begin{equation*}
 \CG_{f_0}g:=\CR\cdot(\CR g+g\CR\CU_0[f_0]+f_0\CR\CU_0[g]).
 \end{equation*}
Since $f_0$ is a critical point of  the Maier-Saupe bulk energy \eqref{maiersaupe}, we have
  \begin{equation*}
    \log f_0+\CU_0[f_0]= \text{const},
  \end{equation*}
  and thus,
   \begin{equation}
    \CR f_0+f_0\CR\CU_0[f_0]=0.\label{eq:f0}
  \end{equation}
A straightforward computation leads to
\begin{equation}\label{eq:1.63}
    \CG_{f_0}g=-\CA_{f_0}\CH_{f_0}g,
  \end{equation}
where $\CA_{f_0}$ and $\CH_{f_0}$ are self-adjoint operators defined by
\begin{equation}\label{eq:1.72}
    \CA_{f_0}\phi=-\CR\cdot(f_0\CR\phi),\quad \CH_{f_0}g=\frac g{f_0}+\CU_0[g].
\end{equation}
In a similar manner, if we define
\begin{equation*}
  \CH_{f_0}^\epsilon h:=\frac h{f_0}+\CU_\epsilon[h],
\end{equation*}
then
\begin{equation*}
  \CG^\ve_{f_0}g:=\CR\cdot(\CR g+g\CR\CU_0[f_0]+f_0\CR\CU_\ve[g])= -\CA_{f_0}\CH^\ve_{f_0}g.
  \end{equation*}
Recall that the kernel space of  $\mathcal{G}_{f_0}$ has been completely characterized in \cite[Theorem 4.6]{ WZZ-cpam}:
\begin{equation}\label{eq:1.64}
  \ker\CG_{f_0}=\ker \CH_{f_0}=\big\{\Theta\cdot\CR f_0:~\Theta\in\RR^3\big\}.
\end{equation}
For any $g\in \mathcal{P}_0(\BS):=\big\{g\in L^2(\BS): \int_{\BS}g(m)\ud m=0\big\}$, we use the following decomposition:
\ben
g=g^\top+g^\bot\in \ker \mathcal{G}_{f_0}\oplus_{f^{-1}_0} (\ker \mathcal{G}_{f_0})^\bot,\quad (\text{i. e. }\int_\BS \tfrac{g^\bot g^\top} {f_0} \ud m=0)\label{decom}
\een
where due to formula \eqref{eq:Rf},
\begin{align}\label{kerout}\nonumber
   (\ker \mathcal{G}_{f_0})^\bot=&~\left\{h\in \mathcal{P}_0(\BS): \int_\BS \tfrac{h \tilde h} {f_0} \ud m=0, ~ \forall \tilde h\in \ker \mathcal{G}_{f_0}\right\}\\
   =&~\left\{h\in \mathcal{P}_0(\BS): \int_\BS  (m\cdot n) \big[(m\wedge n)\cdot\Theta\big] h\ud m=0,~ \forall \Theta\in\mathbb{R}^3\right\}.
\end{align}
In addition, we have the following estimates.
\begin{Lemma}\label{lem:lowerbound}
There exist some $\ve$-independent constants $C_1, C_2>0$ such that
\begin{equation*}
  \begin{split}
     & C_1 \|g^\bot\|^2_{L^2(\BS)}+ \alpha Q[g]:(Q[g]-k_\ve* Q[g]) \leq  \langle \CH^\ve_{f_0} g,g\rangle,\\
&\|g^\top\|^2_{L^2(\BS)}\leq C_2|Q[g]|^2,
  \end{split}
\end{equation*}
{where $\langle \cdot,\cdot\rangle$ denotes the standard inner product in $L^2(\BS)$.}
\end{Lemma}
 \begin{proof}
  Note that $Q[g]:(Q[g]-k_\ve*Q[g])$ may not be positive pointwisely.
 It follows from  \cite[Proposition 4.5]{WZZ-cpam} that
 $$ \langle \CH_{f_0} g,g\rangle\geq C_1\|g^\perp\|^2_{L^2(\BS)}.$$
This together with \eqref{relation:U} and \eqref{def:L} implies
\begin{equation*}
   \langle \CH^\ve_{f_0} g,g\rangle =  \langle \CH_{f_0} g,g\rangle+\langle (\CU_\ve-\CU_0)[g], g\rangle
 \ge  C_1 \|g^\bot\|^2_{L^2(\BS)}+\alpha Q[g]:(Q[g]-k_\ve*Q[g]),
\end{equation*}
which gives the first inequality.
To prove the second one, we can assume $n=(0,0,1)^T$ without loss of generality. In this case, we have $f_0=\frac {e^{\eta m_3^2}}Z$ and
$$\ker \mathcal{G}_{f_0}=\operatorname{span}\big\{m_1m_3f_0, m_2m_3f_0\big\},$$
From \eqref{kerout}, we have $Q_{13}[g^\bot]=Q_{23}[g^\bot]=0$. Thus
\begin{equation*}
  {\|g^\top\|^2_{L^2(\BS)}}\le C(|Q_{13}[g^\top]|^2+|Q_{23}[g^\top]|^2)=C(|Q_{13}[g]|^2+|Q_{23}[g]|^2)\le C|Q[g]|^2.
\end{equation*}
The proof is completed.
 \end{proof}
The following lemma, proved in \cite{KD, WZZ-cpam}, gives a characterization of the kernel space of the adjoint operator $\mathcal{G}^*_{f_0}$.
\begin{Lemma}\label{lem:kerG}
The limiting equilibrium distribution $f_0$ (obtained in Proposition \ref{prop:compact2}) fulfills
\begin{equation*}
    \ker\mathcal{G}^*_{f_0}= \operatorname{span}\big\{\mathcal{A}_{f_0}^{-1}\CR_if_0 \big\}_{1\leq i\leq 3}.
\end{equation*}
That is, a function $\psi(m)\in \ker\mathcal{G}^*_{f_0}$ if and only if there exists $\Theta\in \mathbb{R}^3$ such that
\begin{equation*}
-\CR\cdot (f_0\CR \psi)=\Theta\cdot\CR f_0.
\end{equation*}
\end{Lemma}
\begin{proof}
{We use  $\langle \cdot,\cdot\rangle$ to denote the standard inner product in $L^2(\BS)$.}
It follows from \eqref{eq:1.63} that
\begin{align}
 \langle\mathcal{G}^*_{f_0}\psi, \phi\rangle=  \langle\psi,\mathcal{G}_{f_0} \phi\rangle=  \langle\psi,-\mathcal{A}_{f_0}\CH_{f_0} \phi\rangle= -\langle\CH_{f_0}\mathcal{A}_{f_0}\psi, \phi\rangle.\nonumber
\end{align}
Thus, $\psi\in \ker\mathcal{G}^*_{f_0}$ if and only if $\CA_{f_0}\psi\in \ker\CH_{f_0}$ and according to \eqref{eq:1.64}, it is equivalent to $\CA_{f_0}\psi=\Theta\cdot\CR f_0$ for some $\Theta\in\mathbb{R}^3$.
Apparently, $\psi$ is smooth with respect to the variable $m\in \BS$.
\end{proof}

Let us denote
\begin{equation}
h_\ve:=\sqrt{f_\ve}\CR(\log f_\ve+\CU_\ve[f_\ve])=\sqrt{f_\ve}\(\frac{\CR f_\ve}{f_\ve}+ \CR \CU_\ve [f_\ve]\).\nonumber
\end{equation}
It is easy to see that
\begin{equation}\label{eq:1.26}
  \CR\cdot(\CR f_\ve+f_\ve \CR \CU_\ve [f_\ve])=\CR\cdot (\sqrt{f_\ve}h_\ve).
\end{equation}

\begin{Lemma}\label{lem:G-decom}
The difference $g_\ve:=f_\ve-f_0$ fulfills
\begin{equation}
  \CG_{f_0}^\ve g_\ve
    =\CR\cdot\big(\sqrt{f_\ve}h_\ve-f_0\CR(\CU_\ve-\CU_0)[f_0]-g_\ve\CR(\CU_\ve-\CU_0)[f_0]-g_\ve\CR \CU_\ve[g_\ve]\big).\nonumber
\end{equation}
\end{Lemma}
\begin{proof}
Using  \eqref{eq:1.26}, the right hand side of the formula can be written as
\begin{equation*}
  \begin{split}
 &\CR\cdot(\sqrt{f_\ve}h_\ve-f_0\CR(\CU_\ve-\CU_0)[f_0]-g_\ve\CR(\CU_\ve-\CU_0)[f_0]-g_\ve\CR \CU_\ve[g_\ve])\\
&=\CR\cdot(\sqrt{f_\ve}h_\ve-f_0\CR \CU_\ve [f_0]+f_\ve\CR\CU_0[f_0]-g_\ve\CR \CU_\ve [f_\ve])\\
&=\CR\cdot(\CR f_\ve+f_\ve \CR \CU_\ve [f_\ve]-f_0\CR \CU_\ve [f_0]+f_\ve\CR\CU_0[f_0]-g_\ve\CR \CU_\ve [f_\ve])\\
&=\CR\cdot(\CR f_\ve+f_0\CR\CU_\ve [f_\ve]-f_0\CR \CU_\ve [f_0]+f_\ve\CR\CU_0[f_0])\\
&=\CR\cdot(\CR f_\ve+f_0\CR\CU_\ve [g_\ve]+f_\ve\CR\CU_0[f_0]).
      \end{split}
\end{equation*}
On the other hand, we can employ \eqref{eq:f0} to obtain
\begin{align*}
\CG_{f_0}^\ve g_\ve&= \CR\cdot(\CR g_\ve+f_0 \CR \CU_\ve [g_\ve]+g_\ve\CR\CU_0[f_0])\\
&=\CR\cdot(\CR f_\ve+f_0\CR \CU_\ve [g_\ve]+f_\ve\CR\CU_0[f_0]),
\end{align*}
which yields the lemma.
\end{proof}

\subsection{Strong compactness of $f_\ve$}

Now we derive the strong compactness of $f_\ve$ via the energy dissipation estimate in \eqref{push1}.
\begin{Proposition}\label{prop:strong}
For every  $T>0$ and every compact set $W\subseteq \RR^d$, modulo the extraction of a subsequence,
\begin{equation}
  f_\ve  \xrightarrow{\ve\to 0} f_0\quad\text{strongly in}~L^2(\BS\times W\times(0,T)\big).\nonumber
\end{equation}
\end{Proposition}
\begin{proof}
 Let $g_\ve=f_\ve-f_0$. Then it is equivalent to prove $\lim_{\ve\to 0}\|g_\ve\|_{L^2(W\times \BS\times (0,T))}=0$. First recall from \eqref{limit:Qf} that
\begin{align}\label{estimate:Qg}
Q[g_\ve]\xrightarrow{\ve\to 0}0\quad \text{strongly in } C([0,T];L^2 ({W})),
\end{align}
and thus $\lim_{\ve\to 0}\|g^\top_\ve\|_{L^2( \BS\times{W}\times (0,T))}=0$ by Lemma \ref{lem:lowerbound}.
Therefore, we only need to prove
 \begin{equation}\label{eq:1.74}
   \lim_{\ve\to 0}\|g^\bot_\ve\|_{L^2(\BS \times W\times (0,T))}=0.
 \end{equation}
 To this end, it follows from Lemma \ref{lem:lowerbound} that
\begin{align}\nonumber
& C\|g_\ve^\bot\|^2_{L^2(\BS\times W\times (0,T))}    \leq\int_{W\times (0,T)}\langle \CH^\ve_{f_0} g_\ve,g_\ve\rangle
\ud x\ud t-\alpha\int_{W\times (0,T)} Q[g_\ve]:(Q[g_\ve]-k_\ve* Q[g_\ve]) \ud x\ud t.
\end{align}
 By (\ref{estimate:Qg}), the second term on the right hand side will tend to 0 as $\ve\to 0$. Thus, it suffices to estimate the first term. To this end, we employ
Lemma \ref{lem:G-decom} to obtain
\begin{align}\nonumber
&  \langle \CH^\ve_{f_0} g_\ve,g_\ve\rangle  = -\langle\CG_{f_0}^\ve g_\ve, \CA_{f_0}^{-1}g_\ve\rangle\\ \nonumber
   &=-\int_\BS \CR\cdot(\sqrt{f_\ve}h_\ve-f_0\CR(\CU_\ve-\CU_0)[f_0]-g_\ve\CR(\CU_\ve-\CU_0)[f_0]-g_\ve\CR \CU_\ve[g_\ve])(\CA_{f_0}^{-1}g_\ve)\ud m\\\nonumber
   &= -\int_\BS \CR\cdot(\sqrt{f_\ve}h_\ve)(\CA_{f_0}^{-1}g_\ve)\ud m + {\int_\BS\CR\cdot\( f_0\CR(\CU_\ve-\CU_0)[f_0]\)(\CA_{f_0}^{-1}g_\ve)\ud m}  \\\nonumber
   &\quad+ \int_\BS \CR\cdot\(g_\ve\CR(\CU_\ve-\CU_0)[f_0]\)(\CA_{f_0}^{-1}g_\ve)\ud m
   + \int_\BS \CR\cdot\(g_\ve\CR\CU_\ve[g_\ve]\)(\CA_{f_0}^{-1}g_\ve)\ud m\nonumber\\
   &=: I_1+I_2+I_3+I_4. \label{estimate:g}
\end{align}
To estimate $\{I_j\}_{1\leq j\leq 4}$, we need some inequalities. {Since $\int_\BS g_\ve\ud m=0$,  we recall Poincar\'{e}-Wirtinger  inequality and \eqref{eq:inter} that}
\begin{align}\label{ineq:itpl}
\|g_\ve\|_{L^2(\BS)}^2 \le C_1 \|\CR g_\ve\|^2_{L^2(\BS)},\quad \|g_\ve\|_{L^2(\BS)}^2\le  C_1\(1+\|\CR g_\ve\|_{L^2(\BS)} \),
\end{align}
where $C_1$ only depends on $\BS$.
In addition, the estimate \eqref{eq:1.71} gives rise to
\begin{align}\label{estimate:Rg}
\int_0^T\int_W\|\CR g_\ve\|_{L^2(\BS)}^2\ud x\ud t\le C.
\end{align}
The above two estimates will be repeatedly used.

\medskip

\textit{Estimate  of  $I_1$.} First, we have
\begin{align*}
   |I_1|= &\left| \int_\BS\CR\cdot (\sqrt{f_\ve}h_\ve)\CA_{f_0}^{-1}g_\ve\ud m\right|
     = \left|-\int_\BS \sqrt{f_\ve}h_\ve\cdot\CR\CA_{f_0}^{-1}  g_\ve\ud m\right|\\
     \leq &\left( \int_\BS |h_\ve|^2 \ud m\right)^{1/2}\left(\int_\BS f_\ve|\CR\CA_{f_0}^{-1}  g_\ve|^2\ud m\right)^{1/2}.
  \end{align*}
By \eqref{ineq:itpl}, the Sobolev inequality   and definition of $\mathcal{A}_{f_0}$ at \eqref{eq:1.72},
  \begin{align*}
\int_\BS f_\ve|\CR\CA_{f_0}^{-1}  g_\ve|^2\ud m
    =&\int_\BS g_\ve|\CR\CA_{f_0}^{-1}  g_\ve|^2\ud m+\int_\BS f_0|\CR\CA_{f_0}^{-1}  g_\ve|^2\ud m\\
    \le& C \|g_\ve\|_{L^2(\BS)}\|\CR\mathcal{A}_{f_0}^{-1}g_\ve\|^2_{L^4(\BS)}+C\|g_\ve\|_{L^2(\BS)}^2\\
    \le& C\|g_\ve\|_{L^2(\BS)}^3+C\|g_\ve\|_{L^2(\BS)}^2\\
    \le& C\big( 1 +  \|\CR g_\ve\|_{L^2(\BS)}^\frac32+\|\CR g_\ve\|_{L^2(\BS)}\big).
    \end{align*}
The previous two inequalities together imply
\begin{equation*}
|I_1|\leq  C \|h_\ve\|_{L^2(\BS)}\big(1+{\|\CR g_\ve\|_{L^2(\BS)}}\big).
\end{equation*}
Then it follows from \eqref{push1} that
\begin{equation}\label{20150626bound2}
 \frac 1{\ve^2}\int_{\BS\times{\RR^d}\times (0,T)}    |h_\ve|^2\ud m\ud x\ud t\leq C,
\end{equation}
which together with (\ref{estimate:Rg}) gives
\begin{align}\label{estimate:I1}
\int_0^T\int_W|I_1| \ud x\ud t \le  C\|h_\ve\|_{L^2(\BS\times{W}\times (0,T))}\big(1+\|\CR g_\ve\|_{L^2(\BS\times{W}\times (0,T))}\big) \xrightarrow{\ve\to 0} 0.
\end{align}

\textit{Estimate  of  $I_2$.} It follows from \eqref{relation:U} and \eqref{def:L} that
\begin{equation}\label{eq:1.73}
  \CU_\ve[f_0]-\CU_0[f_0]= \alpha(m\otimes m ): (Q[f_0]-Q[f_0]*k_\ve).
\end{equation}
So  integrating by parts and then employing the above formula leads to
\begin{align*}
 |I_2|=&\Big|\int_\BS \CR\cdot( f_0\CR(\CU_\ve[f_0]-\CU_0[f_0]))(\CA_{f_0}^{-1}g_\ve)\ud m\Big|\nonumber\\
   =&\Big|\int_\BS   (\CU_\ve[f_0]-\CU_0[f_0]) g_\ve\ud m\Big|
    =\alpha \big|Q[g_\ve]:(Q[f_0]-Q[f_0]*k_\ve)\big|.
\end{align*}
{From \eqref{limit:Qf} and the  properties of convolution}, we know that
\begin{align*}
\lim_{\ve\to 0}  \|Q[f_0]-Q[f_0]*k_\ve \|_{L^2(W\times (0,T))} =0,
\end{align*}
which together with  (\ref{estimate:Qg}) implies
\begin{align}\label{estimate:I2}
\int_0^T\int_W|I_2| \ud x\ud t \le C\left\|Q[f_0]-Q[f_0]*k_\ve \right\|_{L^2(W\times (0,T))}\|Q[g_\ve]\|_{L^2(W\times (0,T))}\xrightarrow{\ve\to 0} 0.
\end{align}

\textit{Estimate  of  $I_3$.} Using \eqref{eq:1.73}, we get
\begin{equation*}
 \begin{split}
  | I_3|
   &=\Big|\int_\BS g_\ve  \CR (\CU_\ve[f_0]-\CU_0[f_0])\cdot\CR\CA_{f_0}^{-1} g_\ve \ud m\Big|\\
   &\le C\big|Q[f_0]-Q[f_0]* k_\ve\big| \|g_\ve\|_{L^2(\BS)}{\|\CR\mathcal{A}_{f_0}^{-1}g_\ve\|_{L^2(\BS)}}\\
   &\le C\big|Q[f_0]-Q[f_0]* k_\ve\big| \|g_\ve\|_{L^2(\BS)}^2\\
   &\le C\big|Q[f_0]-Q[f_0]* k_\ve\big|{\(1+\|\CR g_\ve \|_{L^2(\BS)}\)},
 \end{split}
\end{equation*}
where in the last step we employed \eqref{ineq:itpl}. Thus we obtain
\begin{equation}\label{estimate:I3}
  \int_0^T\int_{W}|I_3| \ud x\ud t
  \le  C\big\|Q[f_0]-Q[f_0]*k_\ve \big\|_{L^2(W\times (0,T))}{\(1+\|\CR g_\ve \|_{L^2(\BS\times W\times (0,T))}\)}\xrightarrow{\ve\to 0} 0.
\end{equation}

\textit{Estimate  of  $I_4$.} Using \eqref{eq:RU} and \eqref{ineq:itpl}, we can also estimate  $I_3$ in a similar way,
  \begin{equation*}
 |I_4|  =\Big|\int_\BS g_\ve\CR \CU_\ve [g_\ve]\cdot\CR(\CA_{f_0}^{-1}g_\ve)\ud m\Big|\leq C\big|(Q[g_\ve]*k_\ve)\big|{\(1+\|\CR g_\ve\|_{L^2(\BS)}\)}.
  \end{equation*}
Choose a compact subset $V\subset \RR^d$ such that $W\subset B_r\subset B_{2r}\subset V$ for some $r>0$. Then
it follows from  \eqref{limit:Qf} that
\begin{align*}
 \|Q[g_\ve]*k_\ve\|_{L^2({W\times[0,T]})} \le C\|Q[g_\ve]\|_{L^2({V\times[0,T]})} \xrightarrow{\ve\to 0} 0,
\end{align*}
which yields
\begin{align}\label{estimate:I4}
\lim_{\ve\to 0}\int_0^T\int_W|I_4| \ud x\ud t \le C \lim_{\ve\to 0}\left\|Q[g_\ve]*k_\ve \right\|_{L^2(W\times (0,T))}{\(1+\|\CR g_\ve\|_{L^2(\BS\times W\times (0,T))}\)}=0.
\end{align}
Thus \eqref{estimate:g}, \eqref{estimate:I1}, \eqref{estimate:I2}, \eqref{estimate:I3} and \eqref{estimate:I4} together imply \eqref{eq:1.74} and the proof is completed.
\end{proof}

\section{ Completing the proof of Theorem \ref{thm:converge}}

We start with a lemma involving the commutator:
\begin{Lemma}\label{lem:commu-convergence} For any $\varphi,\psi\in C_c^\infty(\RR^d)$ and $T>0$, it holds that
\begin{equation*}
{  [\mathcal{T}_\ve, \varphi(x)] (\psi Q[f_\ve ])\xrightarrow{\ve\to 0} -i\sqrt{\tfrac\mu {{2d}}}[ \nabla,\varphi(x)](\psi Q[f_0])\quad\text{strongly in}~L^\infty\big(0,T;L^2(\RR^d)\big)},
\end{equation*}
where $f_0$ is the limiting equilibrium distribution in Proposition \ref{prop:compact2}.
\end{Lemma}
\begin{proof}
 We have
\begin{equation*}
  \begin{split}
 & ~[\mathcal{T}_\ve, \varphi(x)] (\psi Q[ {f_\ve}])+[i\sqrt{\tfrac \mu {{2d}}} \nabla,\varphi(x)](\psi Q [ f_0])\\
  =&~  {[\mathcal{T}_\ve, \varphi(x)] \(\psi Q [ {f_\ve}]-  \psi Q [f_0]\)}+  [\mathcal{T}_\ve+i\sqrt{\tfrac \mu {2d}} \nabla, \varphi(x)] (\psi Q [f_0])=: I_1+I_2.
  \end{split}
\end{equation*}
The estimate of $I_1$ follows from   the commutator estimate \eqref{convolu1}, Proposition \ref{compactness1} and Proposition \ref{prop:compact2}: there exists constant $C$ depending on $\varphi,\psi$ such that
\begin{equation*}
 \begin{split}
    &\|[\mathcal{T}_\ve, \varphi(x)] \(\psi Q [ f_\ve ]- \psi Q [f_0]\)\|_{C([0,T];L^2(\RR^d))}\\
    &\leq  C\|\psi Q [ f_\ve ]-  \psi Q [f_0]\|_{C([0,T];L^2(\RR^d))} \xrightarrow{\ve\to 0}0.
 \end{split}
\end{equation*}
To treat $I_2$, it follows from \eqref{eq:1.27} that  $\nabla Q[f_0](x)\in L^\infty(0,T;L^2(\RR^d))$ and thus  $\psi Q[f_0]\in L^\infty(0,T;H^1(\RR^d))$. Consequently we
deduce from Lemma \ref{converge-T} that
$$\lim_{\ve\to0}\|I_2\|_{L^\infty(0,T;L^2(\RR^d))}=0$$
 and the proof is completed.
\end{proof}
The following lemma can be found in \cite{EZ,KD} and we give a more detailed proof here.
\begin{Lemma}\label{Appd:lem:product}
For any  fixed vector $u,v\in\mathbb{R}^3$, the following formula holds
  \begin{equation*}
  \int_\BS ( u\cdot\CR f_0)~\CA_{f_0}^{-1}(v\cdot\CR f_0)\ud m=\gamma \Big(u\cdot v-( u\cdot n)(v\cdot n)\Big),
\end{equation*}
where $\gamma=\gamma(\alpha)$ is a positive constant and $f_0=\frac 1Z e^{\eta(m\cdot n(t,x))^2}$.
\end{Lemma}
\begin{proof}
Since the conclusion is made for  fixed $(t,x )$,  we can assume $n(t,x)=(0,0,1)$ without loss of generality.  Set $\psi_0:=\CA_{f_0}^{-1}(v\cdot\CR f_0)$, then $\psi_0$ solves the follow elliptic equation on $\BS$:
\begin{equation}\label{eq:1.77}
-\CR\cdot(f_0\CR{\psi_0})=v\cdot\CR{f_0}.
\end{equation}
It follows from Fredholm alternative that \eqref{eq:1.77} has a unique solution up to a constant.
On the other hand, since $f_0$ is a local equilibrium, according to Proposition \ref{prop:critical},
$$ \log f_0+\CU_0[f_0]\equiv const.$$
The previous two formulas together imply
\begin{equation}\label{eq:1.76}
\Delta_\BS \psi_0-\CR{u_0}\cdot\CR\psi_0=\CR\cdot\CR{\psi_0}-\CR{u_0}\cdot\CR\psi_0=v\cdot\CR{u_0},
\end{equation}
where $f_0$ and $u_0=\CU_0[f_0]$  only depend on $m_3=m\cdot n(t,x)$. If we assume $\theta\in [0,\pi]$ be the angle between $m$ and $n$ and work under spherical coordinate system $(\theta, \phi)$ with  $\phi\in [0,2\pi)$, then
$m=(\sin\theta\cos\phi, \sin\theta\sin\phi, \cos\theta)$  and it follows from \eqref{eq:1.75} that
\begin{equation}\label{eq:1.51}
  \CR u_0=(-\sin\phi,\cos\phi,0)\frac{\ud u_0}{\ud\theta},\quad \CR f_0=(\sin\phi,-\cos\phi, 0)\frac{\ud f_0(\cos\theta)}{\ud \theta}
\end{equation}
and an explicit formula for $\CR\psi_0(\theta,\phi)$ is available. Then we obtain the following identity through straightforward computation
\[\CR{u_0}\cdot\CR\psi_0=\frac{\ud{u_0}}{\ud\theta}\frac{\partial\psi_0}{\partial\theta}.\]
So we can rewrite \eqref{eq:1.76} in terms of spherical coordinate:
\begin{equation}\label{eq:1.78}
\frac{1}{\sin\theta}\frac{\partial}{\partial\theta}\Big(\sin\theta\frac{\partial\psi_0}{\partial\theta}\Big)
+\frac{1}{\sin^2\theta}\frac{\partial^2{\psi_0}}{\partial\phi^2}
-\frac{\ud{u_0}}{\ud\theta}\frac{\partial\psi_0}{\partial\theta}=v\cdot\ee_\phi\frac{\ud{u_0}}{\ud\theta},
\end{equation}
where $\ee_\phi:=-\frac{m\wedge n}{|m\wedge n|}=(-\sin\phi, \cos\phi, 0)\in\BS$. If  we plug the ansatz
\begin{equation}\label{eq:1.79}
\psi_0(\theta,\phi)=-v\cdot\ee_\phi{g_0}(\cos\theta)=v\cdot(-\sin\phi, \cos\phi, 0){g_0}(\cos\theta)
\end{equation}
into \eqref{eq:1.78}, then   $g_0$ satisfies  the following ODE \cite{KD}:
\begin{equation*}
\frac{1}{\sin\theta}\frac{\ud}{\ud\theta}\Big(\sin\theta\frac{\ud{g_0}}{\ud\theta}\Big)
-\frac{g_0}{\sin^2\theta}
-\frac{\ud{u_0}}{\ud\theta}\frac{\ud{g_0}}{\ud\theta}=-\frac{\ud{u_0}}{\ud\theta}.
\end{equation*}

Then we compute using \eqref{eq:1.51} and \eqref{eq:1.79}
\begin{equation}\label{eq:1.41}
  \begin{split}
     &\int_\BS ( u\cdot\CR f_0)~\CA_{f_0}^{-1}(v\cdot\CR f_0)\ud m\\
=&-  \int_\BS u\cdot(\sin \phi,-\cos\phi,0)\frac{\ud{f_0(\cos\theta)}}{\ud\theta} v\cdot(\sin \phi,-\cos\phi,0) g_0(\cos\theta)\ud m\\
=&-\int_0^{\pi}\sin\theta\int_0^{2\pi}\sin\theta(u_1\sin\varphi-u_2\cos\varphi)f_0'(\cos\theta)(v_1\sin\varphi
-v_2\cos\varphi)g_0(\cos\theta)\ud\theta\ud\varphi\\
=&\frac{1}{2}(u_1v_1+u_2v_2)2\pi\int_0^{\pi}\frac{\ud{f_0(\cos\theta)}}{\ud\theta}g_0(\cos\theta)\sin\theta\ud\theta={\gamma} u\cdot\big[v-(v\cdot n) n],
  \end{split}
\end{equation}
which concludes our lemma with $${\gamma=\pi\int_0^{\pi}\frac{\ud{f_0(\cos\theta)}}{\ud\theta}g_0(\cos\theta)\sin\theta\ud\theta}.$$ Note that $\gamma$ is a constant only depending on $\alpha$.
Thanks to the positivity of $\CA_{f_0}$  and hence $\CA_{f_0}^{-1}$  on $\mathcal{P}_0(\BS)$, by choosing $u=v$ in \eqref{eq:1.41}, we infer  $\gamma > 0$ since   $\CR f_0$ can not be zero on $\BS$ when $\alpha>7.5$.
\end{proof}
\smallskip

{\it Proof of  Theorem \ref{thm:converge}  Part (ii)}.
In the sequel,  we choose any $\Theta(x)\in C^\infty_c(\RR^d;\RR^3)$ and $\varphi(t)\in C_c^\infty(\RR_+;\RR)$ and
use $\psi(m,x,t):=\varphi(t)\CA_{f_0}^{-1}\big(\Theta(x)\cdot\CR f_0\big)$ as a test function for \eqref{main:eqn}. Denote ${\Omega= \RR^d\times  \RR_+}$. Then we have
\begin{align*}
\int_{\BS\times\Omega}\p_tf_\ve(m,x,t)\psi(m,x,t) \ud m\ud x\ud t
= \frac 1\ve\int_{\BS\times\Omega}\CR\cdot(f_\ve \CR \mu_\ve[f_\ve])\psi(m,x,t)\ud m\ud x\ud t.
\end{align*}
On the other hand, we have for almost every $(x,t)\in\RR^d\times \RR_+$ that
\begin{align*}
&\frac 1\ve\int_\BS\CR\cdot(f_\ve \CR \mu_\ve[f_\ve])\psi\ud m
=-\frac 1\ve\int_\BS  \CR \mu_\ve[f_\ve] \cdot f_\ve\CR\psi \ud m=\frac 1\ve\int_\BS   \mu_\ve[f_\ve] \CR\cdot (f_\ve\CR\psi) \ud m\\
=& -\frac 1\ve\int_\BS { \varphi(t) }\mu_\ve[f_\ve] \Theta(x)\cdot \CR f_\ve \ud m+\frac 1\ve\int_\BS   \mu_\ve [f_\ve]\CR\cdot (f_\ve\CR\psi+{\varphi(t)\Theta(x)} f_\ve) \ud m.
\end{align*}
The previous two identities together imply
\begin{multline}\label{eq:limit}
 \int_{\BS\times\Omega}\p_t f_\ve(m,x,t)\psi(m,x,t) \ud m\ud x\ud t=-\frac 1\ve\int_{\BS\times\Omega}\varphi(t)\mu_\ve[f_\ve] \Theta(x)\cdot \CR f_\ve\ud m\ud x\ud t\\
+\frac 1\ve\int_{\BS\times\Omega}\varphi(t)\mu_\ve [f_\ve]\CR\cdot (f_\ve\CR\psi+{\varphi(t)\Theta(x)} f_\ve)\ud m\ud x\ud t.
\end{multline}
Now, we claim the following facts:
\begin{align}\label{limit-1}
 &\lim_{\ve\to0}  \int_{\BS\times\Omega}\p_t f_\ve(m,x,t)\psi(m,x,t)\ud m\ud x\ud t={\gamma}\int_{\Omega}(\partial_t n\wedge n)\cdot \Theta(x)\varphi(t)\ud x\ud t,\\
&\lim_{\ve\to 0}-\frac 1\ve \int_{\O}\varphi(t)\Theta (x) \cdot\int_\BS   \mu_\ve[f_\ve]  \CR f_\ve\ud m\ud x\ud t=\frac{2S_2^2\alpha\mu}{d}\varepsilon^{k\ell i}\int_{\O}\varphi(t)n_{\ell}\nabla\Theta_k(x) \cdot \nabla n_{i}\ud x\ud t,   \label{limit-2}
\\ \label{limit-3}
&\lim_{\ve\to 0}\frac 1\ve\int_{\BS\times\Omega}\mu_\ve[f_\ve] \CR\cdot \big(f_\ve\CR\psi(m,x,t)+{\varphi(t)}\Theta(x) f_\ve\big)\ud m\ud x\ud t = 0.
\end{align}
Here $\gamma=\gamma(\alpha)>0$ is defined in Lemma \ref{Appd:lem:product}.\smallskip

Assuming (\ref{limit-1})-\eqref{limit-3}, we have
\[
\begin{split}
  {\gamma}\int_{\RR^d\times\RR_+} (\partial_t n\wedge n)\cdot \Theta(x)\varphi(t)\ud x\ud t =&\frac{2\alpha\mu S_2^2} d \varepsilon^{k\ell i}\int_{\RR^d\times\RR_+}\varphi(t)n_{\ell}\partial_j\Theta_k(x)  \partial_j n_{i}\ud x\ud t\\
=&\frac{2\alpha\mu S_2^2} d \int_{\RR^d\times\RR_+} \varphi(t)\partial_j\Theta(x)\cdot( n\wedge\partial_jn)\ud x\ud t,
\end{split}
\]
which implies that $n(x,t)$ is a weak solution to the harmonic map heat flow
\begin{align*}
n\wedge(\partial_tn-\Lambda\Delta n)=0
   \end{align*}
with $\Lambda=\frac{2\alpha\mu S_2^2} {\gamma d}>0$. To recover the initial data, we employ  \eqref{limit:Qf} at $t=0$ and get
\begin{equation*}
 Q[f^{in}_\ve]\xrightarrow{\ve\to 0} Q[{f_0}]|_{t=0}=S_2(n(x,0)\otimes n(x,0)-\tfrac 13 I_3) \text{ strongly in }~ L_{loc}^2(\RR^{d}).
\end{equation*}

Next we prove the claims  (\ref{limit-1})-\eqref{limit-3}.\smallskip

{\underline{\it Proof of (\ref{limit-1})}}. It follows from Proposition \ref{prop:compact2} that
 \begin{equation*}
  \lim_{\ve\to0}\int_{\BS\times\Omega}\p_t f_\ve(m,x,t)\psi(m,x,t)\ud m\ud x\ud t=\int_{\BS\times\Omega}\p_t f_0(m,x,t)\psi(m,x,t)\ud m\ud x\ud t.
\end{equation*}
Using the fact $f_0(m,x,t)=\frac1Z e^{\eta (m\cdot n(x,t))^2}$ and (\ref{eq:Rf}), we obtain
\begin{equation*}
  \begin{split}
    \CR f_0&=\frac{1}{Z}e^{\eta (m\cdot n(t,x))^2}2\eta ( m\wedge n)(n\cdot m),\\
  \partial_tf_0&=\frac{1}{Z}e^{\eta (m\cdot n(t,x))^2}2\eta (m\cdot\partial_tn)(m\cdot n)
=(\partial_t n\wedge n)\cdot\CR f_0.
  \end{split}
\end{equation*}
Thus, by Lemma \ref{Appd:lem:product}, it holds that
\begin{align*}
 \int_{\BS\times\Omega}\p_t f_0(m,x,t)\psi(m,x,t) \ud m\ud x\ud t
=&\int_{\BS\times\Omega}\varphi(t)(\partial_t n\wedge n)\cdot\CR f_0~~\CA_{f_0}^{-1}(\Theta(x)\cdot\CR f_0)\ud m\ud x\ud t\\
=&{\gamma}\int_{\RR^d\times \RR_+} (\partial_t n\wedge n)\cdot \Theta(x)\varphi(t)\ud x\ud t,
\end{align*}
which gives  \eqref{limit-1}.\smallskip

{\underline{\it Proof of (\ref{limit-2})}}. We deduce from \eqref{eq:RU} that, for $k\in\{1,2,3\}$ and almost every $(x,t)\in\O$,
\begin{align*}
    -\frac 1\ve\int_\BS   \mu_\ve[f_\ve]  \CR_{k} f_\ve\ud m
    =& \frac 1\ve\int_\BS \CR_k  \mu_\ve[f_\ve]   f_\ve\ud m\\
    =& \frac 1\ve\int_\BS\CR_k f_\ve\ud m+ \frac 1\ve\int_\BS f_\ve\CR_k \CU_\ve[f_\ve] \ud m\\
    =&-\frac{2\alpha}\ve \int_\BS f_\ve   m_\ell m_j\varepsilon^{\ell ik}Q_{ij}[f_\ve]*k_\ve\ud m\\
    =&-\frac{2\alpha}\ve  \varepsilon^{k \ell i} Q_{\ell j}[f_\ve]Q_{ij}[f_\ve]*k_\ve\\
    =&\frac{2\alpha}\ve  \varepsilon^{k \ell i} Q_{\ell j}[f_\ve]\(Q_{ij}[f_\ve]-Q_{ij}[f_\ve]*k_\ve\)\\
     =& 2\alpha   \varepsilon^{k \ell i} Q_{\ell j}[f_\ve]\mathcal{L}_\ve Q_{ij}[f_\ve].
\end{align*}
To proceed, we choose ${\phi(x)}\in C_0^\infty(\RR^d)$ such that $\phi(x)\equiv 1$ on a bounded  open set $V$ such that
 ${\rm{supp}}~\Theta \subset V.$
 As a result, there is a constant $\delta_1>0$ such that
\begin{equation}\label{gap1}
  \rm{dist}(\rm{supp} ~\Theta ,\rm{ supp}(1-\phi))\geq \delta_1>0.
\end{equation}
Therefore, we have
\begin{align*}
    &-\frac 1\ve \int_{\RR^d}  \Theta (x) \cdot\int_\BS   \mu_\ve[f_\ve]  \CR f_\ve\ud m\ud x\\
    =&~2\alpha\int_{\RR^d} \Theta_k(x) \varepsilon^{k \ell i} Q_{\ell j}[f_\ve]\mathcal{L}_\ve Q_{ij}[f_\ve]\ud x \\
     =&~2\alpha\int_{\RR^d} \Theta_k(x) \varepsilon^{k \ell i} Q_{\ell j}[f_\ve]\mathcal{L}_\ve \big( {\phi}(x) Q_{ij}[f_\ve]\big)\ud x\\
     &+2\alpha\int_{\RR^d} \Theta_k(x) \varepsilon^{k \ell i} Q_{\ell j}[f_\ve]\mathcal{L}_\ve\big((1- {\phi}(x)) Q_{ij}[f_\ve]\big)\ud x=:L^\ve_1+L^\ve_2.
\end{align*}

According to our choice of $\phi$, $L_1^\ve$ can be written as
\begin{align*}
    L_1^\ve
    =&~2\alpha\int_{\RR^d}{ \phi(x)}  \Theta_k(x) \varepsilon^{k \ell i} Q_{\ell j}[f_\ve]\mathcal{L}_\ve \big(\phi(x) Q_{ij}[f_\ve]\big)\ud x \\
    =&-2\alpha\varepsilon^{k \ell i}\int_{\RR^d} \mathcal{T}_\ve
    \Big(\Theta_k(x)  \phi(x) Q_{\ell j}[f_\ve] \Big)\cdot  \mathcal{T}_\ve (\phi(x) Q_{ij}[f_\ve])\ud x \\
    =&-{2\alpha\varepsilon^{k \ell i}\int_{\RR^d} [\mathcal{T}_\ve,\Theta_k(x) ] ( \phi Q_{\ell j}[f_\ve]) \cdot   \mathcal{T}_\ve (\phi Q_{ij}[f_\ve])\ud x }.
\end{align*}
By Lemma \ref{lem:commu-convergence}, Lemma \ref{variation5} and the construction of $\psi$, we can pass $\ve\to 0$ in the above identity to obtain
\begin{equation}\label{eq:1.52}
  \begin{split}
    &\lim_{\ve\to 0} {\int_{\RR_+} \varphi(t) L_1^\ve\ud t} \\
    &=-2\alpha\varepsilon^{k\ell i}\int_{\RR^d\times \RR_+} \left[-i\sqrt{\tfrac\mu {{2d}}}\nabla,\Theta_k(x) \right]  \phi Q_{\ell j}[f_0]  \cdot  \left(-i\sqrt{\tfrac\mu  {{2d}}}\right)\nabla (\phi  Q_{ij}[f_0])\varphi(t)\ud x\ud t \\
    &=\frac{\alpha\mu} d \varepsilon^{k\ell i}\int_{\RR^d\times \RR_+} [\nabla,\Theta_k(x) ]  Q_{\ell j}[f_0] \cdot\nabla      Q_{ij}[f_0]\varphi(t)\ud x\ud t \\
    &={\frac{S_2^2\alpha\mu}{d}}\varepsilon^{k\ell i}\int_{\RR^d\times \RR_+}\varphi(t)n_{\ell}\nabla\Theta_k(x) \cdot \nabla n_{i}\ud x\ud t.
  \end{split}
\end{equation}
Here we have used the fact that $Q_{ik}[f_0]=S_2(n_in_k-\frac13\delta_{ik})$.

It remains to show $L_2^\ve\xrightarrow{\ve\to 0}0$. To this end, we use  \eqref{gap1} and \eqref{def:L} to rewrite
\begin{equation*}
  \begin{split}
   L_2^\ve&= 2\alpha\int_{\RR^d} \Theta_k(x) \varepsilon^{k\ell i} Q_{\ell j}[f_\ve](x)\mathcal{L}_\ve\((1-\phi(x)) Q_{ij}[f_\ve]\)\ud x\\
   &= \frac{2\alpha}\ve\int_{\RR^d} \Theta_k(x) \varepsilon^{k\ell i} Q_{\ell j}[f_\ve] (x)\Big(  (1-\phi) Q_{ij}[f_\ve]-((1-\phi) Q_{ij}[f_\ve])*k_\ve\Big)\ud x\\
   &= -\frac{2\alpha}\ve\int_{\RR^d} \Theta_k(x) \varepsilon^{k\ell i} Q_{\ell j}[f_\ve](x) \Big( \int_{\RR^d}\big(1-\phi(y)\big) Q_{ij}[f_\ve](y) k_\ve(x-y)\ud y\Big)\ud x.
  \end{split}
\end{equation*}
In view of \eqref{gap1} and \eqref{assump2}, we have
\begin{equation}\label{eq:1.53}
  {|L_2^\ve|\leq \frac{C(\alpha,\Theta,\phi)}\ve\sup_{x\in\rm{supp}(\Theta)}\int_{|x-y|\geq \delta_1}k_\ve(x-y)\ud y\xrightarrow{\ve\to 0}0.}
\end{equation}
This together with \eqref{eq:1.52}  implies \eqref{limit-2}.
\smallskip

{\underline{\it Proof of (\ref{limit-3})}}.
  We denote  ${W_{\delta,T}:=(\delta,T)\times {\rm{supp}}~\Theta(x) }$ and assume  that $\text{supp } \varphi(t)\subseteq (0,T)$.
By \eqref{eq:1.29} and Proposition \ref{prop:strong}, we have
\begin{equation*}
 \sup_{m\in\BS} |f_\ve(m,x,t)-f_0(m,x,t)|\xrightarrow{\ve\to 0} 0~\text{strongly in}~L^2(W_{\delta,T}),\qquad\forall \delta>0.
\end{equation*}
Therefore, it follows from Egorov's theorem that,  for any $\tilde{\ve}>0$, there exists $\delta >0$ and a measurable set $U\subset W_{\delta,T}$ such that $|U|+|W_{0,\delta}|\leq  \tilde{\ve} $ and modulo the extraction of a subsequence,
\begin{equation}\label{eq:1.30}
  f_\ve  \xrightarrow{\ve\to 0} f_0~\text{uniformly on}~ (W_{\delta,T}\backslash U)\times\BS.
\end{equation}
By Lemma \ref{lem:kerG}, we have
$$\CR\cdot(f_0\CR \psi(m,x,t)+f_0\varphi(t)\Theta(x))\equiv 0.$$
Thus we may write the left-hand side of \eqref{limit-3} as
\begin{equation*}
  \begin{split}
  &\frac 1\ve\int_{{\BS\times\RR^d\times \RR_+ }} \mu_\ve[f_\ve]
  \CR\cdot \big(f_\ve(\CR\psi+\varphi(t)\Theta(x))\big)\ud m \ud x\ud t\\
  &=\frac 1\ve\int_{{\BS\times (W_{\delta,T}\cup W_{0,\delta})}} \mu_\ve[f_\ve]
  \CR\cdot \big(f_\ve(\CR\psi+\varphi(t)\Theta(x))\big)\ud m \ud x\ud t\\
&=\frac 1\ve\int_{ \BS\times {(U\cup W_{0,\delta})} } \mu_\ve[f_\ve] \CR\cdot \big(f_\ve(\CR\psi+{\varphi(t)}\Theta(x))\big)\ud m \ud x\ud t\\
  &\quad+\frac 1\ve\int_{ \BS\times (W_{\delta,T}\backslash U)} \mu_\ve[f_\ve] \CR\cdot \big((f_\ve-f_0)(\CR\psi+{\varphi(t)}\Theta(x))\big)\ud m \ud x\ud t\\
   &=-\frac 1\ve\int_{\BS \times (U\cup W_{0,\delta})} \sqrt{f_\ve}\CR\mu_\ve[f_\ve] \cdot\big(\sqrt{f_\ve}(\CR\psi+{\varphi(t)}\Theta(x))\big)\ud m \ud x\ud t \\
   &\quad- \frac 1\ve\int_{ \BS\times (W_{\delta,T}\backslash U)} \sqrt{f_\ve}\CR\mu_\ve [f_\ve]\cdot(\CR\psi+{\varphi(t)}\Theta(x))\frac{f_\ve-f_0}{\sqrt{f_\ve}}\ud m \ud x\ud t\\
   &=:{J}^\ve_1+{J}^\ve_2.
  \end{split}
\end{equation*}
It remains to show that ${J}_1^\ve$ and ${J}_2^\ve$  both vanish as $\ve\to 0$. For ${J}_1^\ve$, it follows from the Cauchy-Schwarz inequality, \eqref{20150626bound2} and the uniform integrability of $\{f_\ve\}_{\ve>0}$ (see Proposition \ref{prop:compact2}) that
 \begin{equation*}
   |{J}^\ve_1|^2\leq C
   \int_{ \BS\times (U\cup W_{0,\delta})}f_\ve|\CR\psi+\varphi(t)\Theta(x)|^2\ud m \ud x\ud t
 \end{equation*}
 and it  can be made sufficiently small provided that $\tilde{\ve}\ll 1$.
To estimate ${J}^\ve_2$, applying Cauchy-Schwarz inequality, \eqref{20150626bound2} and  \eqref{eq:1.30} yields
 \begin{equation*}
  \begin{split}
     |{J}^\ve_2|^2\leq &~\frac C{\ve^2} \int_{ \BS\times (W_{\delta,T}\backslash U)}f_\ve|\CR \mu_\ve[f_\ve]|^2\ud m \ud x\ud t
     \int_{ \BS\times (W_{\delta,T}\backslash U)}\frac{|f_\ve-f_0|^2}{f_\ve}|\CR\psi{+}\varphi\Theta|^2\ud m \ud x\ud t \\
     \leq & ~C \int_{ \BS\times (W_{\delta,T}\backslash U)}\frac{|f_\ve-f_0|^2}{f_\ve}|\CR\psi{+}\varphi(t)\Theta(x)|^2\ud m \ud x\ud t\xrightarrow{\ve\to 0}0
  \end{split}
 \end{equation*}
 and this proves \eqref{limit-3}.\qed

\section*{Acknowledgments}
The authors would like to thank Professor Zhifei Zhang and Kelei Wang for helpful discussions.
Y. Liu is supported by NSF of China under Grant 11601334.
W. Wang is partly supported by NSF of China under Grant 11501502 and ``the Fundamental Research Funds for the Central Universities" No. 2016QNA3004.

\end{document}